\documentclass[11pt]{amsart}
\usepackage{amsmath,amscd,amssymb}
\usepackage[all,2cell,dvips,cmtip,matrix,arrow]{xy}

\newcommand\sign{\on{sign}}
\newtheorem{theorem}{Theorem}[section]
\newtheorem{corollary}[theorem]{Corollary}
\newtheorem{proposition}[theorem]{Proposition}

\newtheorem{lemma}[theorem]{Lemma}
\theoremstyle{definition}    
\newtheorem{definition}[theorem]{Definition}
\theoremstyle{remark}

\newtheorem{remark}[theorem]{Remark}
\newtheorem{remarks}[theorem]{Remarks}
\newtheorem{example}[theorem]{Example}
\newtheorem{examples}[theorem]{Examples}
\newcommand{\ts}{\textstyle}
\newcommand{\da}{\dasharrow}
\newcommand{\pt}{\operatorname{pt}}
\newcommand{\cox}{\mathsf{h}^\vee}
\newcommand\A{\mathcal{A}}

\newcommand{\K}{\mathbb{K}}

\newcommand{\V}{\mathcal{V}}
\newcommand{\UU}{\mathbb{U}}
\newcommand{\VV}{\mathbb{V}}

\newcommand{\XX}{\mathbb{X}}
\renewcommand{\L}{\mathcal{L}}

\newcommand{\T}{\mathbb{T}}

\newcommand{\Co}{\mathcal{C}}

\newcommand{\ca}{\mathcal}

\newcommand{\U}{\on{U}}
\newcommand{\E}{\ca{E}}
\renewcommand{\S}{\mathcal{S}}
\newcommand{\X}{\mathcal{X}}
\newcommand{\F}{\mathcal{F}}

\newcommand{\R}{\mathbb{R}}
\newcommand{\C}{\mathbb{C}}

\newcommand{\Z}{\mathbb{Z}}

\renewcommand{\P}{\ca{P}}
\newcommand{\Gr}{\on{Gr}}
\newcommand\lie[1]{\mathfrak{#1}}

\newcommand{\g}{\lie{g}}

\newcommand{\on}{\operatorname}

\newcommand{\Ad}{ \on{Ad} }

\newcommand{\ad}{\on{ad}}

\newcommand{\res}{{\on{res}}}
\renewcommand{\ker}{ \on{ker}}

\newcommand{\Spin}{ \on{Spin}}

\newcommand{\SO}{ \on{SO}}

\newcommand{\tpi}{{2\pi\i}}
\renewcommand{\i}{i}

\newcommand\qu{/\kern-.7ex/} 
\newcommand{\fus}{{\on{fus}}}
\newcommand{\red}{{\on{red}}}

\newcommand{\hra}{\hookrightarrow}

\renewcommand{\d}{{\mbox{d}}}
\newcommand{\ol}{\overline}

\newcommand\Phinv{\Phi^{-1}}

\newcommand\Lag{\on{Lag}}
\newcommand\Sig{\Sigma}
\newcommand\sig{\sigma}
\newcommand\eps{\epsilon}
\newcommand\Om{\Omega}
\newcommand\om{\omega}

\newcommand{\f}{\frac}

\newcommand{\p}{\partial}
\renewcommand{\l}{\langle}
\renewcommand{\r}{\rangle}
\newcommand\hh{{\f{1}{2}}}
\newcommand{\ti}{\tilde}
\newcommand{\eeq}{\end{eqnarray*}}
\newcommand{\beq}{\begin{eqnarray*}}

\renewcommand{\H}{\ca{H}}

\newcommand{\Cl}{{\C\on{l}}}
\newcommand{\pr}{\on{pr}}
\newcommand{\wh}{\widehat}
\newcommand{\wt}{\widetilde}
\newcommand{\mf}{\mathfrak}
\begin{document}

\title[]{Dirac structures and\\ Dixmier-Douady bundles}
\author{A. Alekseev}
\address{University of Geneva, Section of Mathematics,
2-4 rue du Li\`evre, 1211 Gen\`eve 24, Switzerland}
\email{alekseev@math.unige.ch}

\author{E. Meinrenken}
\address{University of Toronto, Department of Mathematics,
40 St George Street, Toronto, Ontario M4S2E4, Canada }
\email{mein@math.toronto.edu}
\date{\today}
\begin{abstract}
  A Dirac structure on a vector bundle $V$ is a maximal isotropic
  subbundle $E$ of the direct sum $V\oplus V^*$. We show how to
  associate to any Dirac structure a Dixmier-Douady bundle $\A_E$,
  that is, a $\Z_2$-graded bundle of $C^*$-algebras with typical fiber
  the compact operators on a Hilbert space. The construction has good
  functorial properties, relative to Morita morphisms of
  Dixmier-Douady bundles. As applications, we show that the
  Dixmier-Douady bundle $\A_G^{\on{Spin}}\to G$ over a compact,
  connected Lie group (as constructed by Atiyah-Segal) is
  multiplicative, and we obtain a canonical `twisted
  $\Spin_c$-structure' on spaces with group valued moment maps.
\end{abstract}

\maketitle

\begin{quote}
  {\it \small Dedicated to Richard Melrose on the occasion of his 60th
    birthday.}
\end{quote}
\vskip1cm

\maketitle
\tableofcontents

\section{Introduction}
A classical result of Dixmier and Douady \cite{di:ch} states that
 the degree three cohomology group $H^3(M,\Z)$ classifies 
Morita isomorphism classes of $C^*$-algebra bundles $\A\to M$, with
typical fiber $\K(\H)$ the compact operators on a Hilbert space. Here a
Morita isomorphism $\E\colon \A_1\da \A_2$ is a bundle $\E\to M$ of bimodules ,
locally modeled on the $\K(\H_2)-\K(\H_1)$ bimodule $\K(\H_1,\H_2)$.
Dixmier-Douady bundles $\A\to M$ may be regarded as higher analogues
of line bundles, with Morita isomorphisms 
replacing line bundle isomorphisms. An important example of
a Dixmier-Douady bundle is the Clifford algebra bundle of a Euclidean vector
bundle of even rank; a Morita isomorphism $\Cl(V)\da \C$ amounts to a
$\Spin_c$-structure on $V$.

Given a Dixmier-Douady bundle $\A\to M$, one has the twisted
$K$-homology group $K_0(M,\A)$, defined as the $K$-homology of the
$C^*$-algebra of sections of $\A$ (see Rosenberg \cite{ros:co}).
Twisted $K$-homology is a covariant functor relative to morphisms
\[(\Phi,\E)\colon \A_1\da \A_2,\] 
given by a proper map $\Phi\colon M_1\to M_2$ and a Morita isomorphism
$\E\colon \A_1\da \Phi^*\A_2$. For example, if $M$ is an
even-dimensional Riemannian manifold, the twisted $K$-group
$K_0(M,\Cl(TM))$ contains a distinguished \emph{Kasparov fundamental
class} $[M]$, and in order to push this class forward under the map
$\Phi\colon M\to \pt$ one needs a Morita morphism 
$\Cl(TM)\da \C$,
i.e. a $\Spin_c$-structure on $M$. The push-forward $\Phi_*[M]\in
K_0(\pt)=\Z$ is then the index of the associated $\Spin_c$-Dirac
operator. Similarly, if $\A\to G$ is a Dixmier-Douady bundle over a
Lie group, the definition of a `convolution product' on $K_0(G,\A)$ as
a push-forward under group multiplication $\on{mult}\colon G\times
G\to G$ requires an associative Morita morphism $(\on{mult},\E)\colon
\pr_1^*\A\otimes\pr_2^*\A\da \A$.

In this paper, we will relate the Dixmier-Douady theory to
Dirac geometry. A (linear) \emph{Dirac structure} $(\VV,E)$ over $M$
is a vector bundle $V\to M$ together with a subbundle 
\[ E\subset \VV:=V\oplus V^*,\] such that $E$ is maximal isotropic relative to the
natural symmetric bilinear form on $\VV$. Obvious examples of Dirac
structures are $(\VV,V)$ and $(\VV,V^*)$. 

One of the main results of this paper is the construction of a
\emph{Dirac-Dixmier-Douady functor}, associating to any Dirac
structure $(\VV,E)$ a Dixmier-Douady bundle $\A_E$, and to every
`strong' morphism of Dirac structures $(\VV,E)\da (\VV',E')$ a Morita
morphism $\A_E\da \A_{E'}$.

The Dixmier-Douady bundle $\A_{V^*}$ is canonically Morita trivial,
while $\A_V$ (for $V$ of even rank) is canonically Morita isomorphic to
$\Cl(V)$. An interesting example of a Dirac structure is the
Cartan-Dirac structure $(\T G,E)$ for a compact Lie group $G$. 
The Cartan-Dirac structures is multiplicative, in the sense that there
exists a distinguished Dirac morphism
\begin{equation}\label{eq:dirac1}
 (\T G,E)\times (\T G,E)\da (\T G,E)
\end{equation}
(with underlying map the group multiplication). The
associated Dixmier-Douady bundle $\A_E=:\A_G^{\on{spin}}$ is related
to the spin representation of the loop group $LG$. This bundle (or
equivalently the corresponding bundle of projective Hilbert spaces)
was described by Atiyah-Segal \cite[Section 5]{at:twi}, and plays a
role in the work of Freed-Hopkins-Teleman \cite{fr:lo2}.  As an
immediate consequence of our theory, the Dirac morphism \eqref{eq:dirac1}
gives rise to a Morita morphism
\begin{equation}\label{eq:morita}
(\on{mult},\E)\colon \pr_1^*\A_G^{\on{Spin}}\otimes \pr_2^*\A_G^{\on{Spin}}\da \A_G^{\on{Spin}}.
\end{equation}

Another class of examples comes from the theory of quasi-Hamiltonian
$G$-spaces, that is, spaces with $G$-valued moment maps $\Phi\colon
M\to G$ \cite{al:mom}. Typical examples of such spaces are products of
conjugacy classes in $G$.  As observed by Bursztyn-Crainic
\cite{bur:di}, the structure of a
quasi-Hamiltonian space on $M$ defines a strong Dirac morphism $(\T
M,TM)\da (\T G,E)$ to the Cartan-Dirac structure. Therefore, our theory 
gives a Morita morphism $\A_{TM}\da \A_G^{\on{Spin}}$. On the other
hand, as remarked above $\A_{TM}$ is canonically Morita isomorphic to 
the Clifford bundle $\Cl(TM)$, provided  $\dim M$ is even (this
is automatic if $G$ is connected). One may think of the resulting 
Morita morphism
\begin{equation}\label{eq:spin}
 \Cl(TM)\da  \A_G^{\on{Spin}}
\end{equation}
(with underlying map $\Phi$) as a `twisted $\Spin_c$-structure' on $M$ (following the terminology
of Bai-Lin Wang \cite{wan:ge} and Douglas \cite{dou:twi}).
In a forthcoming paper \cite{me:prep}, we will define a
\emph{pre-quantization} of $M$ \cite{sha:pre,xu:mom} in terms of a
$G$-equivariant Morita morphism $(\Phi,\E)\colon \C\da
\A^{\on{preq}}_G$.  Tensoring with \eqref{eq:spin}, one obtains a
push-forward map in equivariant twisted $K$-homology
\[\Phi_*\colon K_0^G(M,\Cl(TM))\to
K_0^G(G,\A^{\on{preq}}_G\otimes \A_G^{\on{Spin}}).\] 
For $G$ compact, simple and simply connected, the
Freed-Hopkins-Teleman theorem \cite{fr:lo1,fr:lo2} identifies the
target of this map as the fusion ring (Verlinde algebra) $R_k(G)$,
where $k$ is the given level.  The element $\mathcal{Q}(M)=\Phi_*[M]$
of the fusion ring will be called the \emph{quantization} of the
quasi-Hamiltonian space. We will see in \cite{me:prep} that its
properties are similar to the geometric quantization of Hamiltonian
$G$-spaces.

The organization of this paper is as follows. In Section \ref{sec:dir}
we consider Dirac structures and morphisms on vector bundles, and some
of their basic examples.
We observe that any Dirac morphism defines
a path of Dirac structures inside a larger bundle. We introduce
the `tautological' Dirac structure over the orthogonal group and show
that group multiplication lifts to a Dirac morphism. Section
\ref{sec:dix} gives a quick review of some Dixmier-Douady theory.  In
Section \ref{sec:fam} we give a detailed construction of 
Dixmier-Douady bundles from families of skew-adjoint real Fredholm
operators. In Section \ref{sec:ddd} we observe that any Dirac structure on a
Euclidean vector bundle gives such a family of skew-adjoint real
Fredholm operators, by defining a family of boundary conditions for
the operator $\f{\p}{\p t}$ on the interval $[0,1]$. Furthermore, to any Dirac morphism we
associate a Morita morphism of the Dixmier-Douady bundles, and we show
that this construction has good functorial properties. In Section
\ref{sec:qham} we describe the construction of twisted
$\Spin_c$-structures for quasi-Hamiltonian $G$-spaces. In Section
\ref{sec:loop}, we show that the associated Hamiltonian loop group
space carries a distinguished `canonical line bundle', generalizing
constructions from \cite{fr:lg} and \cite{me:can}.

\vskip.2in

{\bf Acknowledgments.} It is a pleasure to thank Gian-Michel Graf,
Marco Gualtieri and Nigel Higson for useful comments and discussion.
Research of A.A. was supported by the grants 200020-120042 and
200020-121675 of the Swiss National Science Foundation. 
E.M was supported by an NSERC Discovery Grant and a Steacie
Fellowship.

\section{Dirac structures and Dirac morphisms}\label{sec:dir}
We begin with a review of linear Dirac structures on vector spaces and
on vector bundles \cite{al:pur, bur:gauge}. In this paper, we will not
consider any notions of integrability.
\subsection{Dirac structures}
For any vector space $V$, the direct sum $\VV=V\oplus V^*$ carries a
non-degenerate symmetric bilinear form extending the pairing between
$V$ and $V^*$,
\[ \l x_1,x_2\r=\mu_1(v_2)+\mu_2(v_1),\ \ \ x_i=(v_i,\mu_i).\]
A \emph{morphism} $(\Theta,\om)\colon \VV\da \VV'$ is a linear map
$\Theta\colon V\to V'$ together with a 2-form $\om\in \wedge^2 V^*$.
The composition of two morphisms $(\Theta,\om)\colon \VV\da \VV'$ and
$(\Theta',\om')\colon \VV'\da \VV''$ is defined as follows:
\[ (\Theta',\om')\circ (\Theta,\om)=(\Theta'\circ\Theta,\om+\Theta^*\om').\]
Any morphism $(\Theta,\om)\colon \VV\da \VV'$ defines a relation between
elements of $\VV,\VV'$ as follows:
\[ (v,\alpha)\sim_{(\Theta,\om)}(v',\alpha')\ \Leftrightarrow\
v'=\Theta(v),\ \alpha=\iota_v\om+\Theta^*\alpha'.\]
Given a subspace $E\subset \VV$, we define its \emph{forward image} to
be the set of all $x'\in \VV'$ such that $x\sim_{(\Theta,\om)} x'$ for
some $x\in E$. For instance, $V^*$ has forward image equal to $(V')^*$.
Similarly, the \emph{backward image} of a subspace $E'\subset \VV'$ is the set
of all $x\in \VV$ such that $x\sim_{(\Phi,\om)} x'$ for some $x'\in
E'$. The backward image of $\{0\}\subset \VV'$ is denoted
$\ker(\Theta,\om)$, and the forward image of $\VV$ is denoted
$\on{ran}(\Theta,\om)$.

A subspace $E$ is called \emph{Lagrangian} if it is maximal isotropic,
i.e. $E^\perp=E$. Examples are $V,V^*\subset \VV$.  
The forward image of a Lagrangian subspace $E\subset \UU$ under a
Dirac morphism $(\Theta,\om)$ is again Lagrangian. On the set of
Lagrangian subspaces with $E\cap \ker(\Theta,\om)=0$, the forward
image depends continuously on $E$.  The choice of a Lagrangian
subspace $E\subset \VV$ defines a (linear) \emph{Dirac structure},
denoted $(\VV,E)$ . We say that $(\Theta,\om)$ defines a \emph{Dirac}
morphism
\begin{equation}\label{eq:dirac}
 (\Theta,\om)\colon (\VV,E)\da
(\VV',E')\end{equation}
if $E'$ is the \emph{forward image} of $E$, and a
\emph{strong} Dirac morphism if furthermore $E\cap
\ker(\Theta,\om)=0$. 
The composition of strong Dirac morphisms is again a
strong Dirac morphism.
\begin{examples}\label{ex:mor}
\begin{enumerate}
\item Every morphism $(\Theta,\om)\colon \VV\da \VV'$ defines a strong
      Dirac morphism $(\VV,V^*)\da (\VV',(V')^*)$.
\item The zero Dirac morphism $(0,0)\colon (\VV,E)\da (0,0)$ is strong
      if and only if $E\cap V=0$. 
\item\label{it:sym} Given vector spaces $V,V'$, any 2-form $\om\in
      \wedge^2 V^*$ defines a Dirac morphism
      $(0,\om)\colon (\VV,V)\da (\VV',(V')^*)$. It is a \emph{strong}
      Dirac morphism if and only if $\omega$ is non-degenerate. (This
      is true in particular if $V'=0$.)
\item If $E=V$, a Dirac morphism $(\Theta,\om)\colon (\VV,V)\da
      (\VV',E')$ is strong if and only if $\ker(\om)\cap
      \ker(\Theta)=0$. 
\end{enumerate}
\end{examples}

\subsection{Paths of Lagrangian subspaces}\label{subsec:morhom}
The following observation will be used later on. Suppose
\eqref{eq:dirac} is a strong Dirac morphism. 
Then there is a distinguished path connecting the subspaces
\begin{equation}\label{eq:values}
E_0=E\oplus (V')^*,\ \  E_1=V^*\oplus E',\ \ 
\end{equation}
of $\VV\oplus \VV'$, as follows. Define a family of morphisms
$(j_t,\om_t)\colon \VV\da \VV\oplus \VV'$ interpolating between
$(\on{id}\oplus 0,0)$ and $(0\oplus \Theta,\om)$:
\[ j_t(v)=((1-t)v,t\Theta(v)),\ \ \om_t=t\om.\]
Then 
\[ \ker(j_t,\om_t)=\begin{cases} 0 & t\not=1,\\
\ker(\Theta,\om) & t=0.
\end{cases}\]
Since $(\Theta,\om)$ is a strong
Dirac morphism, it follows that $E$ is transverse to $\ker(j_t,\om_t)$
for all $t$. Hence the forward images $E_t\subset \VV\oplus \VV'$ under
$(j_t,\om_t)$ are a continuous path of Lagrangian subspaces, taking
on the values \eqref{eq:values} for $t=0,1$. We will refer to $E_t$ as
the \emph{standard path} defined by the Dirac morphism \eqref{eq:dirac}. 

Given another strong
Dirac morphism $(\Theta',\om')\colon (\VV',E')\da (\VV'',E'')$, define
a 2-parameter family of morphisms $(j_{tt'},\om_{tt'})\colon
\VV\da \VV\oplus \VV'\oplus \VV''$ by
\[ j_{tt'}(v)=\big((1-t-t')v,t\Theta(v),t'\Theta'(\Theta(v))\big),\ \ \ 
\om_{tt'}=t\om+t'(\om+\Theta^*\om')\]
Then 
\[ \ker(j_{tt'},\om_{tt'})=\begin{cases} 0 & t+t'\not=1\\
  \ker(\Theta,\om) & t+t'=1,\ t\not=0,\\ 
\ker((\Theta',\om')\circ(\Theta,\om)),& t=0,\ t'=1\end{cases}.\] 
In all cases, $\ker(j_{tt'},\om_{tt'})\cap E=0$, hence we obtain a
continuous 2-parameter family of Lagrangian subspaces $E_{tt'}\subset 
\VV\oplus \VV'\oplus \VV''$ by taking the forward images of $E$. 
We have,
\[ E_{00}=E\oplus (V')^*\oplus (V'')^*,\ \ E_{10}=V^*\oplus E'\oplus (V'')^*,\ E_{01}=V^*\oplus
(V')^*\oplus E''\] 
Furthermore, the path $E_{s0}$ (resp. $E_{0s}$, $E_{1-s,s}$) is  
the direct sum of $(V'')^*$ (resp. of $(V')^*$, $V^*$)
with the standard path defined by $(\Theta,\om)$
(resp. by $(\Theta',\om')\circ(\Theta,\om)$, $(\Theta',\om')$.)

\subsection{The parity of a Lagrangian subspace}
Let $\on{Lag}(\VV)$ be the \emph{Lagrangian Grassmannian} of $\VV$,
i.e. the set of Lagrangian subspaces $E\subset \VV$. It is a
submanifold of the Grassmannian of subspaces of dimension $\dim V$.
$\on{Lag}(\VV)$ has two connected components, which are distinguished
by the $\mod 2$ dimension of the intersection $E\cap V$. We will say
that $E$ has \emph{even} or \emph{odd parity}, depending on whether
$\dim(E\cap V)$ is even or odd. The parity is preserved under strong
Dirac morphisms:
\begin{proposition}\label{prop:evenif}
  Let $(\Theta,\om)\colon (\VV,E)\da (\VV',E')$ be a strong Dirac
  morphism. Then the parity of $E'$ coincides with that of $E$. 
\end{proposition}
\begin{proof}
Clearly, $E$ has the same parity as $E_0=E\oplus (V')^*$, while 
$E'$ has the same parity as $E_1=V^*\oplus E'$. But the Lagrangian
subspaces $E_0,E_1\subset \VV\oplus \VV'$ have the same parity since
they are in the same path component of $\on{Lag}(\VV\oplus \VV')$. 
\end{proof}

\subsection{Orthogonal transformations}\label{subsec:orth1}
Suppose $V$ is a Euclidean vector space, with inner product $B$. Then
the Lagrangian Grassmannian $\on{Lag}(\VV)$ is isomorphic to
the orthogonal group of $V$, by the map associating to $A\in \on{O}(V)$ the Lagrangian
subspace
\[ E_A=\{((I-A^{-1})v,\ (I+A^{-1}) \ts\f{v}{2})|\ v\in V\}. \]
Here $B$ is used to identify $V^*\cong V$, and the factor of $\hh$ in
the second component is introduced to make our conventions consistent
with \cite{al:pur}. For instance,
\[ E_{-I}=V,\ \ E_I=V^*,\ \ E_{A^{-1}}=(E_A)^{\on{op}}\]
where we denote $E^{\on{op}}=\{(v,-\alpha)|\ (v,\alpha)\in E\}$.  It
is easy to see that the Lagrangian subspaces corresponding to
$A_1,A_2$ are transverse if and only if $A_1-A_2$ is invertible; more
generally one has $E_{A_1}\cap E_{A_2}\cong \ker(A_1-A_2)$.  As a
special case, taking $A_1=A,\ A_2=-I$ it follows that the parity of a
Lagrangian subspace $E=E_A$ is determined by $\det(A)=\pm 1$.

\begin{remark}
The definition of $E_A$ may also be understood as follows. Let
$V^-$ denote $V$ with the opposite bilinear form $-B$. Then $V\oplus
V^-$ with split bilinear form $B\oplus (-B)$ is
isometric to $\VV=V\oplus V^*$ by the map $(a,b)\mapsto (a-b,\
(a+b)/2)$. This defines an inclusion $\kappa\colon \on{O}(V)\hra
\on{O}(V\oplus {V}^-)\cong \on{O}(\VV)$. The group $\on{O}(\VV)$ acts
on Lagrangian subspaces, and one has $E_A=\kappa(A)\cdot V^*$.  
\end{remark}

\subsection{Dirac structures on vector bundles}
The theory developed above extends to (continuous) vector bundles
$V\to M$ in a straightforward way. Thus, Dirac structures
$(\VV,E)$ are now given in terms of Lagrangian sub-bundles $E\subset
\VV=V\oplus V^*$. Given a Euclidean metric on $V$, the Lagrangian
sub-bundles are identified with sections $A\in\Gamma(\on{O}(V))$.  A Dirac
morphism $(\Theta,\om)\colon (\VV,E)\da (\VV',E')$ is a vector bundle
map $\Theta\colon V\to V'$ together with a 2-form $\om\in
\Gamma(\wedge^2 V^*)$, such that the fiberwise maps and 2-forms define
Dirac morphisms $(\Theta_m,\om_m)\colon (\VV_m,E_m)\da
(\VV'_{\Phi(m)},E'_{\Phi(m)})$.  Here $\Phi$ is the map on the base
underlying the bundle map $\Theta$.

\begin{example}
  For any Dirac structure $(\VV,E)$, let $U:=\on{ran}(E)\subset V$ be
  the projection of $E$ along $V^*$. If $U$ is a sub-bundle of $V$,
  then the inclusion $U\hra V$ defines a strong Dirac morphism,
  $(\UU,U)\da (\VV,E)$. More generally, if $\Phi\colon N\to M$ is such
  that $U:=\Phi^*\on{ran}(E)\subset \Phi^*V$ is a sub-bundle, then
  $\Phi$ together with fiberwise inclusion defines a strong Dirac morphism
  $(\UU,U)\da (\VV,E)$. For instance, if $(\VV,E)$ is invariant under
  the action of a Lie group, one may take $\Phi$ to be the inclusion
  of an orbit.
\end{example}

\subsection{The Dirac structure over the orthogonal group}
Let $X$ be a vector space, and put $\mathbb{X}=X\oplus X^*$. The
trivial bundle $V_{\Lag(\XX)}=\Lag(\XX)\times X$ carries a
\emph{tautological Dirac structure} $(\VV_{\Lag(\XX)},E_{\Lag(\XX)})$,
with fiber $(E_{\Lag(\XX)})_m$ at $m\in \on{Lag}(\mathbb{X})$ the
Lagrangian subspace labeled by $m$.  Given a Euclidean metric $B$ on $X$,
we may identify $\on{Lag}(\mathbb{X})=\on{O}(X)$; the tautological
Dirac structure will be denoted by
$(\VV_{\on{O}(X)},\,E_{\on{O}(X)})$.  It is equivariant for the
conjugation action on $\on{O}(X)$.  We will now show that the
tautological Dirac structure over $\on{O}(X)$ is multiplicative, in
the sense that group multiplication lifts to a strong Dirac morphism.
Let $\Sigma\colon V_{\on{O}(X)}\times V_{\on{O}(X)}\to V_{\on{O}(X)}$ 
be the bundle map, given by the group multiplication on $
V_{\on{O}(X)}$ viewed as a semi-direct product $\on{O}(X)\ltimes
X$. That is, 
\begin{equation}\label{eq:semidirect}
\Sigma((A_1,\xi_1),(A_2,\xi_2))=(A_1A_2,\ A_2^{-1}\xi_1+\xi_2).
\end{equation}
Let $\sig$ be the 2-form on $V_{\on{O}(X)}\times V_{\on{O}(X)}$, 
given at $(A_1,A_2)\in \on{O}(X)\times\on{O}(X)$ as follows, 
\begin{equation}\label{eq:sigmaform}
 \sig_{(A_1,A_2)}((\xi_1,\xi_2),(\zeta_1,\zeta_2))
=\ts{\hh} ( B(\xi_1,\,A_2 \zeta_2)-B(A_2 \xi_2,\zeta_1)).
\end{equation}
Similar to \cite[Section 3.4]{al:pur} we have:
\begin{proposition}\label{prop:mult}
The map $\Sigma$ and 2-form $\sigma$ define a strong Dirac morphism 
\[ (\Sigma,\sig)\colon (\VV_{\on{O}(X)},E_{\on{O}(X)})\times (\VV_{\on{O}(X)},E_{\on{O}(X)})\da (\VV_{\on{O}(X)},E_{\on{O}(X)})\]
This morphism is associative in the sense that 
\[ (\Sig,\sig)\circ (\Sig\times\on{id},\sig\times 0)
=(\Sig,\sig)\circ (\on{id}\times \Sig,0\times \sig)\] as morphisms
$(\VV,E)\times (\VV,E)\times (\VV,E)\da (\VV,E)$.
\end{proposition}
\begin{proof}[Outline of Proof]
Given $A_1,A_2\in \on{O}(X)$ let $A=A_1A_2$, and put 
\begin{equation}\label{eq:basis}
e(\xi)=((I-A^{-1})\xi,\ (I+A^{-1})\ts{\f{\xi}{2}}),\ \ \xi\in X.
\end{equation}
Define $e_i(\xi_i)$ similarly for $A_1,A_2$. One checks that 
\[ e_1(\xi_1)\times e_2(\xi_2)\sim_{(\Sigma,\sig)} e(\xi)\] if and only
if $\xi_1=\xi_2=\xi$.  The straightforward calculation  
is left to the reader. It follows that every element in
$E_{\on{O}(X)}|_A$ is related to a unique element in $
E_{\on{O}(X)}|_{A_1}\times E_{\on{O}(X)}|_{A_2}$. 
\end{proof}

\subsection{Cayley transform and exponential map}\label{subsec:exp}
The trivial bundle $V_{\wedge^2 X}=\wedge^2 X\times X$ carries a Dirac
structure $(\VV_{\wedge^2 X},\ E_{\wedge^2 X})$, with fiber at $a\in
\wedge^2 X$ the graph $\on{Gr}_a=\{(\iota_\mu a,\mu)|\,\mu\in X^*\}$.
It may be viewed as the restriction of the tautological Dirac
structure under the inclusion $\wedge^2 X\hra \Lag(\XX),\ a\mapsto
\Gr_a$. Use a Euclidean metric $B$ on $X$ to identify $\wedge^2
X=\mf{o}(X)$, and write $(\VV_{\mf{o}(X)},\,E_{\mf{o}(X)})$.  The
orthogonal transformation corresponding to the Lagrangian subspace
$\on{Gr}_a$ is given by the Cayley transform $\ts{\f{I+a/2}{I-a/2}}$.
Hence, the bundle map
\[ \Theta\colon 
V_{\mf{o}(X)}\to V_{\on{O}(X)},\ 
(a,\xi)\mapsto (\ts{\f{I+a/2}{I-a/2}},\,\xi)
\] 
together with the zero 2-form define a strong Dirac morphism
\[(\Theta,0)\colon (\VV_{\mf{o}(X)} ,E_{\mf{o}(X)})\da
(\VV_{\on{O}(X)},\ E_{\on{O}(X)}),\] 
with underlying map the Cayley transform. On the other hand, we may also try to lift the exponential
map $\exp\colon \mf{o}(X)\to \on{O}(X)$. Let 
\begin{equation}\label{eq:semiexp}
 \Pi\colon V_{\mf{o}(X)}\to V_{\on{O}(X)},\ 
(a,\xi)\mapsto (\exp(a),\ \ts{\f{I-e^{-a}}{a}}\xi), 
\end{equation}
the exponential map for the semi-direct product $\mf{o}(X)\ltimes X\to
\on{O}(X)\ltimes X$.  Define a 2-form
$\varpi$ on $V_{\mf{o}(X)}$ by 
\begin{equation}
\varpi_a(\xi_1,\xi_2)=-B(\ts{\f{a-\sinh(a)}{a^2}}\xi_1,\, \xi_2).
\end{equation} 
The following is parallel to 
\cite[Section 3.5]{al:pur}.
\begin{proposition}\label{prop:exp}
  The map $\Pi$ and the 2-form $\varpi$ define a Dirac morphism
  \[ (\Pi,-\varpi)\colon (\VV_{\mf{o}(X)},E_{\mf{o}(X)})\da
  (\VV_{\on{O}(X)},E_{\on{O}(X)}).\]  
It is a strong Dirac morphism
  over the open subset $\mf{o}(V)_\natural$ where the exponential map
  has maximal rank.
\end{proposition}
\begin{proof}[Outline of Proof]
Let $a\in\mf{o}(X)$ and $A=\exp(a)$ be given. Let $e(\xi)$ be as in
\ref{eq:basis}, and define $e_0(\xi)=(a\xi,\xi)$. One checks by
straightforward calculation that 
\[ e_0(\xi)\sim_{(\Pi,-\varpi)} e(\xi)\]
proving that $(\Pi,-\varpi)\colon (\VV_{\mf{o}(X)},E_{\mf{o}(X)})\da
(\VV_{\on{O}(X)},E_{\on{O}(X)})$ is a Dirac morphism. Suppose now that
the exponential map is regular at $a$. By the well-known formula for
the differential of the exponential map, this is equivalent to
invertibility of $\Pi_a$. An element of the form $(a\xi,\xi)$ lies in
$\ker(\Theta,\om)$ if and only if $\Pi_a(a\xi)=0$ and
$\xi=\iota_{a\xi}\varpi_a$. The first condition shows $a\xi=0$, and
then the second condition gives $\xi=0$. Hence
$e_0(\xi)\sim_{(\Pi,-\varpi)}0 \Rightarrow \xi=0$. Conversely, if
$\Pi_a$ is not invertible, and $\xi\not=0$ is an element in the
kernel, then $(a\xi,\xi)\sim_{(\Pi,-\varpi)} 0$.
\end{proof}

%
%
%

\section{Dixmier-Douady bundles and Morita morphisms}\label{sec:dix}
We give a quick review of Dixmier-Douady bundles, geared towards
applications in twisted $K$-theory. For more information we refer to
the articles \cite{di:ch,at:twi,ros:co,mat:fra,mat:ind,mat:inddec} and
the monograph \cite{rae:mo}. Dixmier-Douady bundles are also known as
\emph{Azumaya bundles}.

\subsection{Dixmier-Douady bundles}
A \emph{Dixmier-Douady bundle} is a locally trivial bundle $\A\to M$
of $\Z_2$-graded $C^*$-algebras, with typical fiber $\K(\H)$ the
compact operators on a $\Z_2$-graded (separable) complex Hilbert space, and with
structure group $\on{Aut}(\K(\H))=\on{PU}(\H)$, using the strong
operator topology. The tensor product of two such bundles
$\A_1,\A_2\to M$ modeled on $\K(\H_1),\K(\H_2)$ is a Dixmier-Douady
bundle $\A_1\otimes\A_2$ modeled on $\K(\H_1\otimes\H_2)$. For any
Dixmier-Douady bundle $\A\to M$ modeled on $\K(\H)$, the bundle of
opposite $C^*$-algebras $\A^{\on{op}}\to M$ is a Dixmier-Douady bundle
modeled on $\K(\H^{\on{op}})$, where $\H^{\on{op}}$ denotes the
opposite (or conjugate) Hilbert space. 

\subsection{Morita isomorphisms}
A \emph{Morita isomorphism} $\E\colon \A_1\da \A_2$
between two Dixmier-Douady bundles over $M$ is a $\Z_2$-graded bundle
$\E\to M$ of Banach spaces, with a fiberwise $\A_2-\A_1$ bimodule
structure
\[ \A_2\circlearrowright \E \circlearrowleft \A_1\]
that is locally modeled on $\K(\H_2)\circlearrowright \K(\H_1,\H_2)
\circlearrowleft \K(\H_1)$. Here $\K(\H_1,\H_2)$ denotes the
$\Z_2$-graded Banach space of compact operators from $\H_1$ to $\H_2$.
In terms of the associated principal bundles, a Morita isomorphism is
given by a lift of the structure group $\on{PU}(\H_2)\times
\on{PU}(\H_1^{\on{op}})$ of $\A_2\otimes\A_1^{\on{op}}$ to
$\on{PU}(\H_2\otimes\H_1^{\on{op}})$.  The composition of two Morita
isomorphisms $\E\colon \A_1\da \A_2$ and $\E'\colon \A_2\da \A_3$ is
given by $\E'\circ \E=\E'\otimes_{\A_2}\E$, the fiberwise completion
of the algebraic tensor product over $\A_2$. In local trivializations,
it is given by the composition $\K(\H_2,\H_3)\times \K(\H_1,\H_2)\to
\K(\H_1,\H_3)$.

\begin{examples}
\begin{enumerate}
\item 
A Morita isomorphism $\E\colon \C\da \A$ is called a
      \emph{Morita trivialization} of $\A$, and amounts to a Hilbert
      space bundle $\E$ with an isomorphism $\A=\K(\E)$.
\item Any $*$-bundle isomorphism $\phi\colon \A_1\to \A_2$ may be
      viewed as a Morita isomorphism $\A_1\da \A_2$, by taking $\E=\A_2$
      with the $\A_2-\A_1$-bimodule action $x_2\cdot y\cdot x_1=
      x_2\,y\,\phi(x_1)$.
\item For any Morita isomorphism $\E\colon \A_1\da \A_2$ there is an
      \emph{opposite} Morita isomorphism $\E^{\on{op}}\colon
      \A_2\da\A_1$, where $\E^{\on{op}}$ is equal to $\E$ as a real
      vector bundle, but with the opposite scalar multiplication.
      Denoting by $\chi\colon \E\to \E^{\on{op}}$ the anti-linear map
      given by the identity map of the underlying real bundle, the
      $\A_1-\A_2$-bimodule action reads $x_1\cdot\chi(e)\cdot
      x_2=\chi(x_2^*\cdot e\cdot x_1^*)$.
The Morita isomorphism $\E^{\on{op}}$ is
      `inverse' to $\E$, in the sense that there are canonical
      bimodule isomorphisms
\[ \E^{\on{op}}\circ
      \E\cong \A_1,\ \ \ \ \E\circ \E^{\on{op}}\cong \A_2.\] 
\end{enumerate}
\end{examples}

\subsection{Dixmier-Douady theorem}\label{subsec:dd}
The Dixmier-Douady theorem (in its $\Z_2$-graded version)
states that the Morita isomorphism classes of Dixmier-Douady bundles 
$\A\to M$ are classified by elements
\[ \on{DD}(\A)\in H^3(M,\Z)\times H^1(M,\Z_2),\]
called the \emph{Dixmier-Douady class} of $\A$. Write $\on{DD}(\A)=(x,y)$.
Letting $\hat{\A}$ be the Dixmier-Douady-bundle obtained from $\A$ by
forgetting the $\Z_2$-grading, the element $x$ is the obstruction to
the existence of an (ungraded) Morita trivialization $\hat{\E}\colon
\C\da \hat{\A}$.  The class $y$ corresponds to the obstruction of
introducing a compatible $\Z_2$-grading on $\hat{\E}$. In more detail,
given a loop $\gamma\colon S^1\to M$ representing a homology class
$[\gamma]\in H_1(M,\Z)$, choose a Morita trivialization
$(\gamma,\hat{\F})\colon \C\da \wh{\A}$. Then $y([\gamma])=\pm 1$,
depending on whether or not $\hat{\F}$ admits a compatible
$\Z_2$-grading. 
\begin{enumerate}
\item
The opposite Dixmier-Douady
bundle $\A^{\on{op}}$ has class $\on{DD}(\A^{\on{op}})=-\on{DD}(\A)$. 
\item If $\on{DD}(\A_i)=(x_i,y_i),\ i=1,2$, are the
classes corresponding to two Dixmier-Douady bundles $\A_1,\A_2$ over
$M$, then \cite[Proposition 2.3]{at:twi}
\[ \on{DD}(\A_1\otimes\A_2)=(x_1+x_2+\ti{\beta}(y_1\cup y_2),\
y_1+y_2)\]
where $y_1\cup y_2\in H^2(M,\Z_2)$ is the cup product, and
$\ti{\beta}\colon H^2(M,\Z_2)\to H^3(M,\Z)$ is the Bockstein homomorphism.
\end{enumerate}

\subsection{2-isomorphisms}
Let $\A_1,\A_2$ be given Dixmier-Douady bundles over $M$.
\begin{definition}
  A \emph{2-isomorphism} between two Morita isomorphisms
  \[\E,\E'\colon \A_1\da\A_2\] is a continuous bundle isomorphism
  $\E\to \E'$, intertwining the norms, the $\Z_2$-gradings and the
  $\A_2-\A_1$-bimodule structures.
\end{definition}
Equivalently, a 2-isomorphism may be viewed as a trivialization of the
$\Z_2$-graded Hermitian line bundle
  \begin{equation}\label{eq:linebundle}
 L=\on{Hom}_{\A_2-\A_1}(\E,\E')
\end{equation}
given by the fiberwise bimodule homomorphisms. Any two Morita
bimodules are related by \eqref{eq:linebundle} as $\E'=\E\otimes L$.
It follows that the set of 2-isomorphism classes of Morita isomorphisms
$\A_1\da \A_2$ is either empty, or is a a principal
homogeneous space (torsor) for the group $H^2(M,\Z)\times H^0(M,\Z_2)$
of $\Z_2$-graded line bundles.

\begin{example}\label{eq:hom2}
  Suppose the Morita isomorphisms $\E,\E'$are connected by a continuous
  path $\E_s$ of Morita isomorphisms, with $\E_0=\E,\ \E_1=\E'$.  Then
  they are 2-isomorphic, in fact $L_s=\on{Hom}_{\A_2-\A_1}(\E,\E_s)$
  is a path connecting \eqref{eq:linebundle} to the trivial line
  bundle.
\end{example}

\begin{example}\label{ex:hom3}
  Suppose $\A_s,\ s\in [0,1]$ is a continuous family of
  Dixmier-Douady-bundles over $M$, i.e. their union defines a
  Dixmier-Douady bundle $\A\to [0,1]\times M$. Then there exists a
  continuous family of isomorphisms $\phi_s\colon \A_0\to \A_s$, i.e. 
  an isomorphism $\pr_2^*\A_0\cong \A$ of bundles over $[0,1]\times M$. (The
  existence of such an isomorphism is clear in terms of the associated
  principal $\on{PU}(\H)$-bundles.) By composing with $\phi_0^{-1}$ 
if necessary, we may assume $\phi_0=\on{id}$. Any other such family of isomorphisms $\phi_s'\colon
  A_0\to \A_s,\ \phi_0'=\on{id}$ is related to $\phi_s$ by a family $L_s$ of line
  bundles, with $L_0$ the trivial line bundle. We conclude that the
  homotopy of Dixmier-Douady bundles $\A_s$ gives a distinguished
  2-isomorphism class of isomorphisms $\A_0\to \A_1$.
\end{example}

\subsection{Clifford algebra bundles}\label{subsec:clifford}
%
Suppose that $V\to M$ is a Euclidean vector bundle of
rank $n$. A \emph{$\Spin_c$-structure on $V$} is given by an orientation on $V$
together with a lift of the structure group of $V$ from $\SO(n)$ to
$\on{Spin}_c(n)$, where $n=\on{rk}(V)$. According to Connes \cite{con:lon} and Plymen
\cite{ply:st}, this is equivalent to Definition \ref{def:spinc} below in
terms of Dixmier-Douady bundles. 

Recall that if $n$ is even, then the associated bundle of complex
Clifford algebras $\Cl(V)$ is a Dixmier-Douady bundle, modeled on
$\Cl(\R^n)=\on{End}(\wedge \C^{n/2})$. In this case, a
$\Spin_c$-structure may be defined to be  a Morita trivialization $\S\colon
\C\da \Cl(V)$, with $\S$ is the associated
\emph{spinor bundle}. To include the case of odd rank, it is
convenient to introduce 
\[ \wt{V}=V\oplus \R^n,\ \ \ \wt{\Cl}(V):=\Cl(\wt{V}). \]
\begin{definition} \label{def:spinc}
  A \emph{$\Spin_c$-structure} on a Euclidean vector bundle $V$ is a
  Morita trivialization
\[ \wt{\ca{S}}\colon \C\da \wt{\Cl}(V)\] 
The bundle $\wt{\S}$ is called the corresponding \emph{spinor
  bundle}. An isomorphism of two $\Spin_c$-structures is a 2-isomorphism
of the defining Morita trivializations. 
\end{definition}
If $n$ is even, one recovers $\S$ by composing with the Morita
isomorphism $\wt{\Cl}(V)\da \Cl(V)$. The Dixmier-Douady class $(x,y)$
of $\wt{\Cl}(V)$ is the obstruction to the existence of a
$\Spin_c$-structure: In fact $x$ is the third integral Stiefel-Whitney
class $\ti{\beta}(w_2(V))\in H^3(M,\Z)$, while $y$ is the first
Stiefel-Whitney class $w_1(V)\in H^1(M,\Z_2)$, i.e.  the obstruction
to orientability of $V$.

Any two $\Spin_c$-structures on $V$ differ by a $\Z_2$-graded
Hermitian line bundle, and an isomorphism of $\Spin_c$-structures
amounts to a trivialization of this line bundle. Observe that
there is a Morita trivialization
\[\wedge \wt{V}^\C\colon \C\da \wt{\Cl}(V\oplus V)=\wt{\Cl}(V)\otimes \wt{\Cl}(V)\]
defined by the complex structure on $\wt{V}\oplus \wt{V}\cong
\wt{V}\otimes\R^2$. Hence, given a $\Spin_c$-structure, 
we can define the Hermitian line bundle 
\begin{equation}\label{eq:canon}
 K_{\wt{S}}=\on{Hom}_{\wt{\Cl}(V\oplus V)}(\wt{S}\otimes\wt{S},\ \wedge \wt{V}^\C).\end{equation}
(If $n$ is even, one may omit the $\sim$'s.)  This is the
\emph{canonical line bundle} of the $\Spin_c$-structure.  If the
$\Spin_c$-structure on $V$ is defined by a complex structure $J$, then
the canonical bundle coincides with 
$\det(V_-)=\wedge^{n/2}V_-$, where $V_-\subset V^\C$ is the $-\i$ eigenspace of $J$.


\subsection{Morita morphisms}
It is convenient to extend the notion of Morita isomorphisms of
Dixmier-Douady bundles, allowing non-trivial maps on the base.  A
\emph{Morita morphism}
\begin{equation}\label{eq:morphism}
(\Phi,\E)\colon \A_1\da \A_2
\end{equation} 
of bundles $\A_i\to M_i,\ i=1,2$ is a continuous map $\Phi\colon
M_1\to M_2$ together with a Morita isomorphism $\E\colon \A_1\da
\Phi^*\A_2$.  A given map $\Phi$ lifts to such a Morita morphism if
and only if $\on{DD}(\A_1)=\Phi^*\on{DD}(\A_2)$.  Composition of
Morita morphisms is defined as $(\Phi',\E')\circ (\Phi,\E)=(\Phi'\circ
\Phi,\ \ \Phi^*\E'\circ \E)$. If $\E\colon
\C\da \A$ is a Morita trivialization, we can think of
$\E^{\on{op}}\colon \A\da \C$ as a Morita morphism covering the map
$M\to \pt$. As mentioned in the introduction, a Morita morphism
\eqref{eq:morphism} such that $\Phi$ is \emph{proper} induces a
push-forward map in twisted K-homology.

\subsection{Equivariance}
The Dixmier-Douady theory generalizes to the $G$-equi\-va\-ri\-ant
setting, where $G$ is a compact Lie group. $G$-equivariant
Dixmier-Douady bundles over a $G$-space $M$ are classified by
$H^3_G(M,\Z)\times H^1_G(M,\Z_2)$. If $M$ is a point, a
$G$-equivariant Dixmier-Douady bundle $\A\to \pt$ is of the form
$\A=\K(\H)$ where $\H$ is a $\Z_2$-graded Hilbert space with an action
of a central extension $\wh{G}$ of $G$ by $\U(1)$. (It is a well-known
fact that $H^3_G(\pt,\Z)=H^3(BG,\Z)$ classifies such central
extensions.) The definition of $\Spin_c$-structures in terms of Morita
morphisms extends to the
$G$-equivariant in the obvious way.

\section{Families of skew-adjoint real Fredholm operators}\label{sec:fam}
In this Section, we will explain how a continuous family of
skew-adjoint Fredholm operators on a bundle of real Hilbert spaces
defines a Dixmier-Douady bundle. The construction is inspired by ideas
in Atiyah-Segal \cite{at:twi}, Carey-Mickelsson-Murray\cite{car:ind,mic:ger},
and Freed-Hopkins-Teleman \cite[Section 3]{fr:lo2}.

\subsection{Infinite dimensional Clifford algebras}
We briefly review the spin representation for infinite dimensional
Clifford algebras. Excellent sources for this material are the book
\cite{ply:sp} by Plymen and Robinson and the article \cite{ara:bo} by
Araki.

Let $\V$ be an infinite dimensional real Hilbert space, and
$\V^\C$ its complexification. The Hermitian inner product on $\V^\C$
will be denoted $\l\cdot,\cdot\r$, and the complex conjugation map by
$v\mapsto v^*$. 
%
Just as in the finite-dimensional case, one defines the Clifford
algebra $\Cl(\V)$ as the $\Z_2$-graded unital complex algebra with odd
generators $v\in \V$ and relations, $v v=\l v,v\r$.  The Clifford
algebra carries a unique anti-linear anti-involution $x\mapsto x^*$
extending the complex conjugation on $\V^\C$, and a unique norm
$||\cdot||$ satisfying the $C^*$-condition $||x^*x||=||x||^2$. Thus
$\Cl(\V)$ is a $\Z_2$-graded pre-$C^*$-algebra.

A \emph{(unitary) module over $\Cl(\V)$} is a complex $\Z_2$-graded
Hilbert space $\E$ together with a $*$-homomorphism $\varrho\colon
\Cl(\V)\to \L(\E)$ preserving $\Z_2$-gradings.  Here $\L(\E)$ is the
$*$-algebra of bounded linear operators, and the condition on the
grading means that $\varrho(v)$ acts as an odd operator for each $v\in
\V^\C$.  


%
%
We will view $\L(\V)$ (the bounded $\R$-linear operators on $\V$) as
an $\R$-linear subspace of $\L(\V^\C)$. Operators in $\L(\V)$ will be
called \emph{real}. A real skew-adjoint operator $J\in \L(\V)$ is
called an \emph{orthogonal complex structure} on $\V$ if it satisfies
$J^2=-I$. Note $J^*=-J=J^{-1}$, so that $J\in \on{O}(\V)$.

The orthogonal complex structure defines a decomposition $\V^\C=\V_+\oplus \V_-$
into maximal isotropic subspaces $\V_\pm=\on{ker}(J\mp\i)\subset
\V^\C$. Note
$v\in\V_+\Leftrightarrow v^*\in\V_-$.  Define a Clifford action
of $\Cl(\V)$ on $\wedge\V_+$ by the formula
\[ \rho(v)=\sqrt{2}(\eps(v_+)+\iota(v_-)),\]
writing $v=v_++v_-$ with $v_\pm \in \V_\pm$. Here $\eps(v_+)$ denotes
exterior multiplication by $v_+$, while the contraction $\iota(v_-)$
is defined as the unique derivation such that $\iota(v_-)w=\l
v_-^*,w\r$ for $w\in \V^\C\subset\wedge\V^\C$. Passing to the Hilbert
space completion one obtains 
a unitary $\Z_2$-graded Clifford module
\[  \ca{S}_J=\ol{\wedge \V_+},\]
called the \emph{spinor module} or \emph{Fock representation} defined
by $J$.

The equivalence problem for Fock representations was solved by Shale
and Stinespring \cite{sha:spi}. See also \cite[Theorem 3.5.2]{ply:sp}.
\begin{theorem}[Shale-Stinespring]\label{th:equivalence}
  The $\Cl(\V)$-modules $\ca{S}_1,\ca{S}_2$ defined by orthogonal
  complex structures $J_1,J_2$ are unitarily isomorphic (up to
  possible reversal of the $\Z_2$-grading) if and only if
  $J_1-J_2\in\L_{\on{HS}}(\V)$. In this case, the unitary operator
  implementing the isomorphism is unique up to a scalar $z\in\U(1)$.
  The implementer has even or odd parity, according to the parity of
  $\hh\dim\ker(J_1+J_2)\in \Z$.
\end{theorem}
\begin{definition}\cite[p. 193]{se:el}, \cite{fr:lo2} 
  Two orthogonal complex structures $J_1,J_2$ on a real Hilbert space
  $\V$ are called \emph{equivalent} (written $J_1\sim J_2$) if their
  difference is Hilbert-Schmidt. An equivalence class of complex
  structures on $\V$ (resp. on $\V\oplus \R$) is called an even (resp.
  odd) \emph{polarization} of $\V$.
\end{definition}
By Theorem \ref{th:equivalence} the $\Z_2$-graded $C^*$-algebra
$\K(\S_J)$ depends only on the equivalence class of $J$, in the sense
that there exists a canonical identification $\K(\S_{J_1})\equiv
\K(\S_{J_2})$ whenever $J_1\sim J_2$. That is, any polarization 
of $\V$ determines a Dixmier-Douady algebra.

\subsection{Skew-adjoint Fredholm operators}\label{subsec:skew}
Suppose $D$ is a real skew-adjoint (possibly unbounded) Fredholm
operator on $\V$, with dense domain $\on{dom}(D)\subset \V$.  In
particular $D$ has a finite-dimensional kernel, and $0$ is an isolated
point of the spectrum. Let $J_D$ denote the real skew-adjoint operator,
\[ J_D=i\on{sign}({\ts \f{1}{i}}D)\]
(using functional calculus for the self-adjoint operator ${\ts
  \f{1}{i}}D$).  Thus $J_D$ is an orthogonal complex structure on
$\ker(D)^\perp$, and vanishes on $\ker(D)$.  If $\ker(D)=0$, we may
also write $J_D=\f{D}{|D|}$. The same definition of $J_D$ also applies
to complex skew-adjoint Fredholm operators. We have:
\begin{proposition}\label{prop:graf}
  Let $D$ be a (real or complex) skew-adjoint Fredholm operator, and $Q$ a skew-adjoint
  Hilbert-Schmidt operator. Then $J_{D+Q}-J_D$ is Hilbert-Schmidt. 
\end{proposition}
The following simple proof was shown to us by Gian-Michele Graf.
\begin{proof}
  Choose $\eps>0$ so that the spectrum of $D,D+Q$ intersects the set
  $|z|<2\eps$ only in $\{0\}$. Replacing $D$ with $D+i\eps$ if
  necessary, and noting that $J_{D+i\eps}-J_D$ has finite rank, we may
  thus assume that $0$ is not in the spectrum of $D$ or of $D+Q$. One
  then has the following presentation of $J_D$ as a Riemannian
  integral of the resolvent $R_z(D)=(D-z)^{-1}$,
\[ J_D=-\f{1}{\pi}\int_{-\infty}^\infty R_t(D) \d t,\]
convergent in the strong topology. Using a similar expression for
$J_{D+Q}$ and the second resolvent identity
$R_t(D+Q)-R_t(D)=-R_t(D+Q)\,Q\, R_t(D)$,
we obtain 
\[ J_{D+Q}-J_D=\f{1}{\pi}
\int_{-\infty}^\infty R_t(D+Q)\,Q\, R_t(D) \ \d t.\]
Let $a>0$ be
such that the spectrum of $D,\ D+Q$ does not meet the disk $|z|\le a$.
Then $||R_t(D)||,\ ||R_t(D+Q)||\le (t^2+a^2)^{-1/2}$ for all $t\in\R$.
Hence
\[ || R_t(D+Q)\,Q\, R_t(D) ||_{HS}\le \f{1}{t^2+a^2} ||Q||_{HS},\]
using $||AB||_{HS}\le ||A||\, ||B||_{HS}$. Since $\int (t^2+a^2)^{-1}\d t=\pi/a$, we obtain
the estimate
\begin{equation}\label{eq:estimate}
  ||J_{D+Q}-J_D||_{HS} \le \f{1}{a}||Q||_{HS}.\qedhere\end{equation} 
\end{proof}
A real skew-adjoint Fredholm operator $D$ on $\V$ will be called of
\emph{even} (resp. \emph{odd}) \emph{type} if $\ker(D)$ has even (resp.  odd)
dimension. As in \cite[Section 3.1]{fr:lo2}, we associate to any $D$ of even
type the even polarization defined by the orthogonal complex
structures $J\in \on{O}(\V)$ such that $J-J_D$ is Hilbert-Schmidt. For
$D$ of odd type, we similarly obtain
an odd polarization by viewing $J_D$ as an operator on $\V\oplus \R$ 
(equal to $0$ on $\R$). 

Two skew-adjoint real Fredholm operators $D_1,D_2$ on $\V$ will be
called \emph{equivalent} (written $D_1\sim D_2$) if they define the
same polarization of $\V$, and hence the same Dixmier-Douady algebra
$\A$. Equivalently, $D_i$ have the same parity and $J_{D_1}-J_{D_2}$
is Hilbert-Schmidt. In particular, $D\sim D+Q$ whenever $Q$ is a
skew-adjoint Hilbert-Schmidt operator. In the even case, we
can always choose $Q$ so that $D+Q$ is invertible, while in the odd
case we can choose such a $Q$ after passing to $\V\oplus \R$.

\begin{remark}\label{rem:continuity}
  The estimate \eqref{eq:estimate} show that for fixed $D$ (such that
  $D,D+Q$ have trivial kernel), the difference $J_{D+Q}-J_D\in
  \L_{HS}(\X)$ depends continuously on $Q$ in the Hilbert-Schmidt
  norm. On the other hand, it also depends continuously on $D$
  relative to the norm resolvent topology \cite[page 284]{ree:fu}. 
 This follows from the integral
  representation of $J_{D+Q}-J_D$, together with resolvent identities
  such as
\[R_t(D')-R_t(D)=R_t(D')R_1(D')^{-1}\big(R_1(D')-R_1(D)\big) R_1(D)^{-1} R_t(D).\]
giving estimates  $||R_t(D')-R_t(D)||\le
\,(t^2+a^2)^{-1}\,||R_1(D')-R_1(D)||$ for $a>0$ such that the spectrum
of $D,D'$ does not meet the disk of radius $a$. 
\end{remark}

\subsection{Polarizations of bundles of real Hilbert spaces}
Let $\V\to M$ be a bundle of real Hilbert spaces, with typical fiber
$\mathcal{X}$ and with structure group $\on{O}(\mathcal{X})$ (using
the norm topology). A polarization on $\V$ is a family of
polarizations on $\V_m$, depending continuously on $m$. To make this
precise, fix an orthogonal complex structure $J_0\in \on{O}(\X)$, and
let $\L_{\on{res}}(\X)$ be the Banach space of bounded linear
operators $S$ such that $[S,J_0]$ is Hilbert-Schmidt, with norm
$\|S\|+\|[S,J_0]\|_{HS}$. Define
the \emph{restricted orthogonal group} $\on{O}_{\on{res}}(\X)=
\on{O}(\X)\cap \L_{\on{res}}(\X)$, with the subspace topology. It is a
Banach Lie group, with Lie algebra
$\mf{o}_{\on{res}}(\X)=\mf{o}(\X)\cap \L_{\on{res}}(\X)$.  The unitary
group $\on{U}(\mathcal{X})=\on{U}(\mathcal{X},J_0)$ relative to $J_0$,
equipped with the norm topology is a Banach subgroup of
$\on{O}_{\on{res}}(\mathcal{X})$. For more details on
the restricted orthogonal group, we refer to Araki \cite{ara:bo} or
Pressley-Segal\cite{pr:lo}. 

%
\begin{definition}\label{def:polar}
  An even \emph{polarization} of the real Hilbert space bundle
  $\V\to M$ is a reduction of the structure group
  $\on{O}(\mathcal{X})$ to the restricted orthogonal group
  $\on{O}_{\on{res}}(\mathcal{X})$. An odd polarization of $\V$ 
  is an even polarization of $\V\oplus \R$. 
\end{definition}
Thus, a polarization is described by a system of local trivializations
of $\V$ whose transition functions are continuous maps into
$\on{O}_{\on{res}}(\mathcal{X})$. Any global complex structure on $\V$
defines a polarization, but not all polarizations arise in this way.

\begin{proposition}
  Suppose $\V\to M$ comes equipped with a polarization. For $m\in M$
  let $\A_m$ be the Dixmier-Douady algebra defined by the polarization
  on $\V_m$.  Then $\A=\bigcup_{m\in M}\A_m$ is a Dixmier-Douady
  bundle.
\end{proposition}
\begin{proof}
  We consider the case of an even polarization (for the odd case,
  replace $\V$ with $\V\oplus \R$).  By assumption, the bundle $\V$
  has a system of local trivializations with transition functions in
  $\on{O}_{\on{res}}(\mathcal{X})$. Let $\S_0$ be the spinor module
  over $\Cl(\mathcal{X})$ defined by $J_0$, and $\on{PU}(\S_0)$ the
  projective unitary group with the strong operator topology. A
  version of the Shale-Stinespring theorem \cite[Theorem
  3.3.5]{ply:sp} says that an orthogonal transformation is implemented
  as a unitary transformation of $\S_0$ if and only if it lies in
  $\on{O}_{\on{res}}(\mathcal{X})$, and in this case the implementer
  is unique up to scalar. According to Araki \cite[Theorem
  6.10(7)]{ara:bo}, the resulting group homomorphism
  $\on{O}_{\on{res}}(\mathcal{X})\to \on{PU}(\S_0)$ is continuous.
  That is, $\A$ admits the structure group $\on{PU}(\S_0)$ with the
  strong topology. 
\end{proof}

In terms of the principal $\on{O}_{\on{res}}(\X)$-bundle $\P\to M$
defined by the polarization of $\V$, the Dixmier-Douady bundle is an
associated bundle 
\[ \A=\P\times_{\on{O}_{\on{res}}(\X)}\K(\S_0). \] 

\subsection{Families of skew-adjoint Fredholm operators}\label{subsec:family}
Suppose now that $D=\{D_m\}$ is a family of (possibly unbounded) real
skew-adjoint Fredholm operators on $\V_m$, depending continuously on
$m\in M$ in the norm resolvent sense \cite[page 284]{ree:fu}. That is,
the bounded operators $(D_m-I)^{-1}\in \L(\V_m)$ define a continuous
section of the bundle $\L(\V)$ with the norm topology.  The map
$m\mapsto \dim\ker(D_m)$ is locally constant $\mod 2$. The family $D$
will be called of even (resp. odd) type if all $\dim\ker(D_m)$ are even
(resp. odd).
Each $D_m$ defines an even (resp. odd) polarization of $\V_m$, given by the
complex structures on $\V_m$ or $\V_m\oplus \R$ whose difference with
$J_{D_m}$ is Hilbert-Schmidt. 
\begin{proposition}\label{prop:dd}
  Let $D=\{D_m\}$ be a family of (possibly unbounded) real
  skew-adjoint Fredholm operators on $\V_m$, depending continuously on
  $m\in M$ in the norm resolvent sense. Then the corresponding family
  of polarizations on $\V_m$ depends continuously on $m$ in the sense of
  Definition \ref{def:polar}. That is, $D$ determines a polarization
  of $\V$.
\end{proposition}
\begin{proof}
  We assume that the family $D$ is of even type.  (The odd case is
  dealt with by adding a copy of $\R$.) We will show the existence of
  a system of local trivializations \[\phi_\alpha\colon
  \V|_{U_\alpha}=U_\alpha\times \mathcal{X}\] and skew-adjoint Hilbert-Schmidt
  perturbations $Q_\alpha\in
  \Gamma(\L_{HS}(\V|_{U_\alpha}))$ of $D|_{U_\alpha}$, continuous in the Hilbert-Schmidt norm\footnote{The
    sub-bundle $\L_{HS}(\V)\subset \L(\V)$ carries a topology, where a
    sections is continuous at $m\in M$ if its expression in a local
    trivialization of $\V$ near $m$ is continuous. (This is
    independent of the choice of trivialization.)}, so that
\begin{enumerate}
\item[(i)] $\ker(D_m+Q_\alpha|_m)=0$ for all $m\in U_\alpha$, and
\item[(ii)] $\phi_\alpha\circ
      J_{D+Q_\alpha}\circ\phi_\alpha^{-1}=J_0$.
\end{enumerate}
The transition functions $\chi_{\alpha\beta}=\phi_\beta\circ
\phi_{\alpha}^{-1}\colon U_\alpha\cap U_\beta\to \on{O}(\X)$ will then take values in $\on{O}_\res(\X)$:
Indeed, by Proposition \ref{prop:graf} the difference
$J_{D+Q_\beta}-J_{D+Q_\alpha}$ is Hilbert-Schmidt, and (using
\eqref{eq:estimate} and Remark \ref{rem:continuity}) it is a
continuous section
of $\L_{HS}(\V)$ over $U_\alpha\cap U_\beta$. 
Conjugating by $\phi_\alpha$,
and using (ii) it follows that
\begin{equation} \label{eq:transc}
\chi_{\alpha\beta}^{-1}\circ J_0 \circ
\chi_{\alpha\beta}-J_0\ \colon\  U_\alpha\cap U_\beta\to \L(\X)\end{equation}
takes values in Hilbert-Schmidt operators, and is continuous
in the Hilbert-Schmidt norm.  Hence the $\chi_{\alpha\beta}$ are
continuous functions into $\on{O}_\res(\X)$. 

It remains to construct the desired system of local trivializations. 
It suffices to construct such a trivialization near any given $m_0\in
M$. Pick a continuous family of skew-adjoint Hilbert-Schmidt operators
$Q$ so that $\ker(D_{m_0}+Q_{m_0})=0$. (We may even take $Q$ of finite
rank.)  Hence $J_{D_{m_0}+Q_{m_0}}$ is a complex structure.  Choose an
isomorphism $\phi_{m_0}\colon \V_{m_0}\to \X$ intertwining
$J_{D_{m_0}+Q_{m_0}}$ with $J_0$, and extend to a local trivialization
$\phi\colon \V|_U\to U\times \X$ over a neighborhood $U$ of $m_0$. We
may assume that $\ker(D_m+Q_m)=0$ for $m\in U$, defining complex
structures $J_m=\phi_m\circ J_{D_m+Q_m}\circ \phi_m^{-1}$. By
construction $J_{m_0}=J_0$, and hence $||J_m-J_0||<2$ after $U$
is replaced by a smaller neighborhood if necessary. 
By \cite[Theorem 3.2.4]{ply:sp}, Condition (ii) guarantees that
\[ g_m=(I-J_m
  J_0)\ |I-J_m J_0|^{-1}\]
gives a well-defined continuous map $g\colon U\to \on{O}(\mathcal{X})$
with $J_m=g_m\,J_0\,g_m^{-1}$. 
Hence, replacing $\phi$ with $g\circ \phi$ we obtain a local
trivialization satisfying (i), (ii). 
\end{proof}

To summarize: A continuous family $D=\{D_m\}$ of skew-adjoint real
Fredholm operators on $\V$ determines a polarization of $\V$.  The
fibers $\P_m$ of the associated principal $\on{O}_\res(\X)$-bundle
$\P\to M$ defining the polarization are given as the set of
isomorphisms of real Hilbert spaces $\phi_m\colon \V_m\to \X$ such
that $J_0-\phi_m J_{D_m} \phi_m^{-1}$ is Hilbert-Schmidt. In turn, the
polarization determines a Dixmier-Douady bundle $\A\to M$. 

We list some elementary properties of this construction:
\begin{enumerate}
\item Suppose $\V$ has finite rank. Then $\A=\Cl(\V)$ if the rank is
      even, and $\A=\Cl(\V\oplus \R)$ if the rank is odd. In both
      cases, $\A$ is canonically Morita isomorphic to $\wt{\Cl}(V)$.
\item If $\ker(D)=0$ everywhere, the complex structure $J=D |D|^{-1}$ gives a
      global a spinor module $\S$, defining a Morita trivialization
      \[\S\colon \C\da \A.\]
\item If $\V=\V'\oplus \V''$ and $D=D'\oplus D''$, the corresponding
      Dixmier-Douady algebras satisfy $\A\cong \A'\otimes \A''$,
      provided the kernels of $D'$ or $D''$ are even-dimensional. If
      both $D',D''$ have odd-dimensional kernels, we obtain
      $\A\otimes\Cl(\R^2) \cong \A'\otimes\A''$. In any case, $\A$ is
      canonically Morita isomorphic to $\A'\otimes\A''$.
\item Combining the three items above, it follows that if
      $\V'=\ker(D)$ is a sub-bundle of $\V$, then there is a canonical
      Morita isomorphism %
\[ \wt{\Cl}(\V')\da \A.\]
\item Given a $G$-equivariant family of skew-adjoint Fredholm
      operators (with $G$ a compact Lie group) one obtains a
      $G$-Dixmier-Douady bundle. 
%
\end{enumerate}
Suppose $D_1,D_2$ are two families of skew-adjoint Fredholm operators
as in Proposition \ref{prop:dd}. We will call these families
equivalent and write $D_1\sim D_2$ if they define the same
polarization of $\V$, and therefore the same Dixmier-Douady bundle
$\A\to M$. We stress that different polarizations can induce
\emph{isomorphic} Dixmier-Douady bundles, however, the isomorphism is
usually not canonical.

\section{From Dirac structures to Dixmier-Douady bundles}\label{sec:ddd}
We will now use the constructions from the last Section to associate
to every Dirac structure $(\VV,E)$ over $M$ a Dixmier-Douady bundle
$\A_E\to M$, and to every strong Dirac morphism $(\Theta,\om)\colon
(\VV,E)\da (\VV',E')$ a Morita morphism. The construction is functorial
`up to 2-isomorphisms'.

\subsection{The Dixmier-Douady algebra associated to a Dirac structure}
Let $(\VV,E)$ be a Dirac structure over $M$. Pick a Euclidean metric
on $V$, and let $\V\to M$ be the bundle of real Hilbert spaces with
fibers 
\[ \V_m=L^2([0,1],V_m).\] 
Let $A\in \Gamma(\on{O}(V))$ be the orthogonal section corresponding
to $E$, as in Section \ref{subsec:orth1}.  Define a family
$D_E=\{(D_E)_m,\ m\in M\}$ of operators on $\V$, where
$(D_E)_m=\f{\p}{\p t}$ with domain
\begin{equation}\label{eq:domain}
\on{dom}((D_E)_m)=\{f\in \V_m|\ \dot{f}\in
\V_m,\ f(1)=-A_m\,f(0)\}.\end{equation}
The condition that the distributional derivative $\dot{f}$ lies in
$L^2\subset L^1$ implies that $f$ is absolutely continuous; hence the boundary
condition $f(1)=-A_m f(0)$ makes sense. The unbounded operators
$(D_E)_m$ are closed and skew-adjoint (see e.g. \cite[Chapter
VIII]{ree:fu}). By Proposition \ref{prop:resolv} in the Appendix, the
family $D_E$ is continuous in the norm resolvent sense, hence it
defines a Dixmier-Douady bundle $\A_E$ by Proposition \ref{prop:dd}.

The kernel of the operator $(D_E)_m$ is the intersection of
$V_m\subset \V_m$ (embedded as constant functions) with the domain 
\eqref{eq:domain}. That is, 
\[ \ker((D_E)_m)=\ker(A_m+I)=V_m\cap E_m\]
\begin{proposition}\label{prop:cdo}
Suppose $E\cap V$ is a sub-bundle of $V$. Then there is a canonical
Morita isomorphism 
\[ \wt{\Cl}(E\cap V)\da \A_E.\]
In particular there are canonical Morita isomorphisms 
\[ \C\da \A_{V^*},\ \ \ \wt{\Cl}(V)\da \A_V.\]
\end{proposition}
\begin{proof}
  Since $\ker(D_E)\cong E\cap V$ is then a sub-bundle of $\V$, the
  assertion follows from item (d) in Section \ref{subsec:family}.
\end{proof}

\begin{remark}
  The definition of $\A_E$ depends on the choice of a Euclidean metric
  on $V$. However, since the space of Euclidean metrics is
  contractible, the bundles corresponding to two choices are related
  by a canonical 2-isomorphism class of isomorphisms. See Example
  \ref{ex:hom3}.
\end{remark}
\begin{remark}
  The Dixmier-Douady class $\on{DD}(\A_E)=(x,y)$ is an invariant of
  the Dirac structure $(\VV,E)$. It may
  be constructed more directly as follows: Choose $V'$ such that
  $V\oplus V'\cong X\times \R^N$ is trivial. Then $E\oplus (V')^*$
  corresponds to a section of the orthogonal bundle, or equivalently
  to a map $f\colon X\to \on{O}(N)$. The class $\on{DD}(\A_E)$ is the
  pull-back under $f$ of the class over $\on{O}(N)$ whose restriction
  to each component is a generator of $H^3(\cdot,\Z)$ respectively
  $H^1(\cdot,\Z_2)$. (See Proposition \ref{prop:ddclass} below.)
  However, not all classes in $H^3(X,\Z)\times H^1(X,\Z_2)$ are
  realized as such pull-backs.
\end{remark}

The following Proposition shows that the polarization defined by $D_E$
depends very much on the choice of $E$, while it is not affected by
perturbations of $D_E$ by skew-adjoint multiplication operators
$M_\mu$. Let $L^\infty([0,1],\mf{o}(V))$ denote the Banach bundle
with fibers $L^\infty([0,1],\mf{o}(V_m))$.  Its continuous sections
$\mu$ are given in local trivialization of $V$ by continuous maps to
$L^\infty([0,1],\mf{o}(X))$. Fiberwise multiplication by $\mu$ defines
a continuous homomorphism
\[ L^\infty([0,1],\mf{o}(V))\to \L(\V),\ \ \mu\mapsto M_\mu.\]
\begin{proposition}
\begin{enumerate}
\item
Let $E,E'$ be two Lagrangian sub-bundles of $V$. Then $D_E\sim D_{E'}$ 
if and only if $E=E'$. 
\item
Let $\mu\in \Gamma(L^\infty([0,1],\mf{o}(V)))$, defining a continuous
family of skew-adjoint multiplication operators $M_\mu \in \Gamma(\L(\V))$. 
For any Lagrangian sub-bundle $E\subset \VV$ one has
\[ D_E+M_\mu\sim D_E.\]
\end{enumerate}
\end{proposition}
The proof is given in the Appendix, see Propositions \ref{prop:p2} and \ref{prop:p3}.

\subsection{Paths of Lagrangian sub-bundles}\label{subsec:hom}
Suppose $E_s,\ s\in [0,1]$ is a path of Lagrangian sub-bundles of $V$,
and $A_s\in \Gamma(\on{O}(V))$ the resulting path of orthogonal
transformations. In Example \ref{ex:hom3}, we remarked that there is a path of
isomorphisms $\phi_s\colon \A_{E_0}\to \A_{E_s}$ with
$\phi_0=\on{id}$, and the 2-isomorphism class of the resulting
isomorphism $\phi_1\colon \A_{E_0}\to \A_{E_1}$ does not depend 
on the choice of $\phi_s$. It is also clear from the discussion 
in Example \ref{ex:hom3} that the isomorphism defined by a
concatenation of two paths is 2-isomorphic to the composition of the 
isomorphisms defined by the two paths. 

If the family $E_s$ is differentiable, there is a distinguished choice
of the isomorphism $\A_{E_0}\to \A_{E_1}$, as follows.  
\begin{proposition}\label{prop:hom}
  Suppose that $\mu_s:=-\f{\p A_s}{\p s} A_s^{-1}$ defines a
  continuous section of $L^\infty([0,1],\mf{o}(V))$.  Let $M_\gamma\in
  \Gamma(\on{O}(\V))$ be the orthogonal transformation given fiberwise
  by pointwise multiplication by $\gamma_t=A_t A_0^{-1}$.  Then
\[ M_\gamma\circ D_{E_0}\circ M_\gamma^{-1}=D_{E_1}+M_\mu\sim D_{E_1}.\]
Thus  $M_\gamma$ induces an isomorphism $\A_{E_0}\to \A_{E_1}$.     
\end{proposition}
\begin{proof}
  We have
\[ f(1)=-A_0 f(0)\ \ \Leftrightarrow\ \ (M_\gamma f)(1)=-A_1 (M_\gamma f)(0),\]
which shows $M_\gamma(\on{dom}(D_{E_0}))=\on{dom}(D_{E_1})$, and \[A_t
A_0^{-1} \f{\p}{\p t} (A_0 A_t^{-1} f)=\f{\p f}{\p t}+\mu_t f.\]
\end{proof}

\begin{examples}\label{ex:paths}
\begin{enumerate}
\item\label{it:1} Suppose $E$ corresponds to $A=\exp(a)$ with $a\in
      \Gamma(\mf{o}(V))$. Then $A_s=\exp(s a)$ defines a path
      from $A_0=I$ and $A_1=A$. Hence we obtain an isomorphism
      $\A_{V^*}\to \A_E$. (The 2-isomorphism class of this isomorphism
      may depend on the choice of $a$.)
%
\item Any 2-form $\omega\in \Gamma(\wedge^2 V^*)$ defines an
      orthogonal transformation of $\VV$, given by $(v,\alpha)\mapsto
      (v,\alpha-\iota_v\omega)$. Let $E^\om$ be the image of the
      Lagrangian subbundle $E\subset\VV$ under this transformation.
      The corresponding orthogonal transformations $A,A^\om$ are
      related by 
\[A^\om=(A-\om (A-I))(I-\om(A-I))^{-1},\] 
where we identified the 2-form $\om$ with the corresponding
skew-adjoint map $\om\in\Gamma(\mf{o}(V))$.  Replacing $\omega$ with
$s\omega$, one obtains a path $E_s$ from $E_0=E$ to $E_1=E^\om$,
defining an isomorphism $\A_E\to \A_{E^\om}$.
\end{enumerate}
\end{examples}

\subsection{The Dirac-Dixmier-Douady functor}\label{subsec:functor}
Having assigned a Dixmier-Douady bundle to every Dirac structure on a
Euclidean vector bundle $V$, 
\begin{equation}\label{eq:31} (\VV,E)\ \ \ \leadsto \A_E\end{equation}
we will now associate a Morita morphism to every strong Dirac morphism:  
\begin{equation}\label{eq:32} 
\Big((\Theta,\om)\colon (\VV,E)\da (\VV',E')\Big) \ \ \leadsto\ \ 
\Big((\Phi,\E)\colon \A_E\da \A_{E'}\Big).\end{equation}
Here $\Phi\colon M\to M'$ is underlying the map on the base.  Theorem
\ref{th:functor} below states that \eqref{eq:32} is compatible with
compositions `up to 2-isomorphism'.  Thus, if we take the morphisms
for the category of Dixmier-Douady bundles to be the 2-isomorphism
classes of Morita morphisms, and if we include the Euclidean metric on
$V$ as part of a Dirac structure, the construction \eqref{eq:31},
\eqref{eq:32} defines a functor.  We will call this the
\emph{Dirac-Dixmier-Douady functor}.

The Morita isomorphism $\E\colon \A_E\da \Phi^*\A_{E'}=\A_{\Phi^*E'}$
in \eqref{eq:32} is defined as a composition
\begin{equation}\label{eq:compo} \A_E \da \A_{E\oplus \Phi^* (V')^*}\cong \A_{V^*\oplus
\Phi^* E'}\da \A_{\Phi^* E'},\end{equation}
where the middle map is induced by the path $E_s$ from $E_0=E\oplus
\Phi^* (V')^*$ to $E_1=V^*\oplus \Phi^* E'$, constructed as in Subsection
\ref{subsec:morhom}.  By composing with the Morita isomorphisms
$\A_E\da \A_{E\oplus \Phi^* (V')^*}$ and $\A_{V^*\oplus \Phi^* E'}\da
\A_{E'}$ this gives the desired Morita morphism $\A_E\da \A_{E'}$.
\begin{theorem}\label{th:functor}
i) The composition of the Morita morphisms $\A_E\da \A_{E'}$ and
$\A_{E'}\da \A_{E''}$ defined by two strong Dirac morphisms $
(\Theta,\om)$ and $(\Theta',\om')$ is 2-isomorphic to the Morita
morphism $\A_E\da \A_{E''}$ defined by $(\Theta',\om')\circ
(\Theta,\om)$. ii) The Morita morphism $\A_E\da \A_E$ defined by 
the Dirac morphism $(\on{id}_V,0)\colon (\VV,E)\da (\VV,E)$ is 2-isomorphic to the identity.  
\end{theorem}
\begin{proof}
  i) By pulling everything back to $M$, we may assume that $M=M'=M''$
  and that $\Theta,\Theta'$ induce the identity map on the base. As in
  Section \ref{subsec:morhom}, consider the three Lagrangian subbundles
\[ E_{00}=E\oplus (V')^*\oplus W^*,\ \ E_{10}=V^*\oplus E'\oplus W^*,\ E_{01}=V^*\oplus
(V')^*\oplus E''\] 
of $\VV\oplus \VV'\oplus \VV''$. We have canonical Morita isomorphisms
\[ \A_E\da \A_{E_{00}},\ \ \A_{E'}\da \A_{E_{10}},\ \ \A_{E''}\da
\A_{E_{01}}.\]
The morphism \eqref{eq:compo} may be equivalently
described as a composition 
\[ \A_E \da \A_{E_{00}}\cong \A_{E_{10}}\da \A_{E'},\]
since the path from $E_{00}$ to $E_{10}$ (constructed as in
Subsection \ref{subsec:morhom}) is just the direct sum of $W^*$ with the
standard path from $E\oplus (V')^*$ to $V^*\oplus E'$.  
Similarly, one describes the morphism $\A_{E'}\da \A_{E''}$ as
\[ \A_{E'} \da \A_{E_{10}}\cong \A_{E_{01}}\da \A_{E''}.\]
The composition of the Morita morphisms $\A_E\da \A_{E'}\da \A_{E''}$
defined by $(\Theta,\om),\ (\Theta',\om')$ is hence given by
\[ \A_E  \da \A_{E_{10}}\cong \A_{E_{01}} \cong \A_{E_{01}}\da \A_{E''}.\]
The composition $\A_{E_{10}}\cong \A_{E_{01}} \cong \A_{E_{01}}$ is
2-isomorphic to the isomorphism defined by the concatenation of
standard paths from $E_{00}$ to $E_{10}$ to $E_{01}$. As observed in 
Section \ref{subsec:morhom} this concatenation is homotopic to
the standard path from $E_{00}$ to $E_{01}$, which defines the
morphism $\A_E\da \A_{E''}$ corresponding to $ (\Theta',\om')\circ
(\Theta,\om)$.

ii) We will show that the Morita morphism $\A_E\da \A_{E_0}\cong
\A_{E_1}\da \A_E$ defined by $(\on{id}_V,0)$ is homotopic to the
identity. Here $E_0=E\oplus V^*,\ E_1=V^*\oplus E$, and the
isomorphism $\A_{E_0}\cong \A_{E_1}$ is defined by the standard path
$E_t$ connecting $E_0,E_1$. By definition, $E_t$ is the forward image of $E$ under the
morphism $(j_t,0)\colon \VV\da \VV\oplus \VV$ where 
\[ j_t \colon V\to V\oplus V,\ y\mapsto
((1-t)y,ty).\] 
It is convenient to replace $j_t$ by the isometry,
\[\ti{j}_t=(t^2+(1-t)^2)^{-1/2}\,j_t.\] 
This is homotopic to $j_t$ (e.g.
by linear interpolation), hence the resulting path $\ti{E}_t$ defines
the same 2-isomorphism class of isomorphisms $\A_{E_0}\to \A_{E_1}$.

The splitting of $V\oplus V$ into $V_t:=\on{ran}(\ti{j}_t)$ and $V_t^\perp$
defines a corresponding orthogonal splitting of $\VV\oplus \VV$. The subspace $\ti{E}_t$ 
is the direct sum of the intersections  
\[ \ti{E}_t\cap
\VV_t^\perp=\on{ann}(V_t)=(V_t^\perp)^*,\ \ \ \ti{E}_t\cap
\VV_t=:\ti{E}_t'.\] 
This defines a Morita
isomorphism 
\[ \A_{\ti{E}_t}\da \A_{\ti{E}_t'}\]
%
%
%

On the other hand, the isometric
isomorphism $V\to V_t$ given by $\ti{j}_t$ extends to an isomorphism 
$\VV\to \VV_t$, taking $E$ to $\ti{E}_t'$. Hence $ \A_{\ti{E}_t'}\cong
\A_E$ canonically. In summary, we obtain a family of Morita isomorphisms
\[ \A_E\da \A_{E_0} \cong \A_{\ti{E}_t}\da \A_{\ti{E}_t'}\cong \A_E.\]
For $t=1$ this is the Morita isomorphism defined by $(\on{id}_V,0)$, while
for $t=0$ it is the identity map $\A_E\to \A_E$.
\end{proof}


\subsection{Symplectic vector bundles}\label{subsec:vect}
Suppose $V\to M$ is a vector bundle, equipped with a fiberwise
symplectic form $\om\in \Gamma(\wedge^2 V^*)$. Given a Euclidean
metric $B$ on $V$, the 2-form $\om$ is identified with a skew-adjoint
operator $R_\om$, defining a complex structure $J_\om=R_\om/|R_\om|$
and a resulting spinor module $\S_\om\colon \C\da \Cl(V)$. 
(We may work with $\Cl(V)$ rather than $\wt{\Cl}(V)$, since $V$ has
even rank.) 
\begin{proposition}\label{prop:sympl}
The Morita isomorphism 
\[ S_\om^{\on{op}}\colon \Cl(V)\da \C\]
defined by the $\Spin_c$-structure $S_\om$ is 2-isomorphic to the 
Morita isomorphism  $\Cl(V)\da \A_V$, followed by the 
Morita isomorphism $\A_V\da \C$ defined by the strong Dirac morphism 
$(0,\om)\colon (\VV,V)\da (0,0)$ (cf.  Example
\ref{ex:mor}\eqref{it:sym}). 
\end{proposition}

\begin{proof}
Consider the standard path for the Dirac morphism
$(0,\om)\colon (\VV,E)\da (0,0)$, 
\[  E_t=\{((1-t)v,\alpha)|\ t\iota_v\om+(1-t)\alpha=0\}\subset \VV,\]
defining $\A_V=\A_{E_0}\cong \A_{E_1}=\A_{V^*}\da
\C$. The path of orthogonal
transformations defined by $E_t$ is
\[ A_t=\f{t{R_\om}-\hh(1-t)^2}{t{R_\om}+\hh(1-t)^2}.\]
We will replace $A_t$ with a
more convenient path $\ti{A}_t$, 
\[ \ti{A}_t=-\exp(t\pi J_\om).\]
We claim that this is homotopic to $A_t$ with the same endpoints. 
Clearly $A_0=-I=-\ti{A}_0$ and $A_1=I=\ti{A}_1$. By considering the
action on any eigenspace of $R_\om$, one checks that the spectrum of
both $J_\om A_t$ and $J_\om\ti{A}_t$ is contained in the half space 
$\on{Re}(z)\ge 0$, for all $t\in [0,1]$. Hence  
\begin{equation}\label{eq:invert}
J_\om A_t+I,\ \ J_\om\ti{A}_t+I
\end{equation} 
are invertible for all $t\in [0,1]$. The Cayley transform $C\mapsto
(C-I)/(C+I)$ gives a diffeomorphism from the set of all $C\in
\on{O}(V)$ such that $C+I$ is invertible onto the vector space
$\mf{o}(V)$. By using the linear interpolation of the Cayley
transforms one obtains a homotopy between $J_\om A_t,\ J_\om \ti{A}_t$, and
hence of $A_t,\ti{A}_t$. 

By Proposition \ref{prop:hom}, the path $\ti{A}_t$ defines an
orthogonal transformation $M_\gamma\in \on{O}(\V)$, taking the complex
structure $J_0$ for $E_0=V^*$ to a complex structure
$J_1=M_\gamma\circ J_0\circ M_\gamma^{-1}$ in the equivalence class
defined by $D_{E_1}$. Consider the orthogonal decomposition
$\V=\V'\oplus \V''$ with $\V'=\ker(D_V)\cong V$. Let $J''$ be the
complex structure on $\V''$ defined by $D_V$, and put $J'=J_\om$.
Since
\[ M_\gamma\circ D_{V^*}\circ M_\gamma^{-1}=D_V+\pi J_\om.\]
we see that $J_1=J'\oplus J''$, hence
$\S_1=\S'\otimes\S''=\S_\om\otimes \S''$. The Morita isomorphism
$\Cl(V)\da \A_V$ is given by the bimodule $\E=\S'' \otimes \Cl(V)$.
Since $\Cl(\V)=\S_\om\otimes\S_\om^{\on{op}}$,
 it follows that that
$\E=\S'' \otimes \Cl(V)=\S_1\otimes \S_\om^{\on{op}}$, and
\[ \S_1^{\on{op}}\otimes_{\A_V}\E=\S_\om^{\on{op}}.\]
\end{proof}

\section{The Dixmier-Douady bundle over the orthogonal group}
\label{subsec:orth}
\subsection{The bundle $\A_{\on{O}(X)}$}
As a special case of our construction, let us consider the
tautological Dirac structure $(\VV_{\on{O}(X)},E_{\on{O}(X)})$ for a
Euclidean vector space $X$. Let $\A_{\on{O}(X)}$ be the corresponding
Dixmier-Douady bundle; its restriction to $\SO(X)$ will be denoted
$\A_{\SO(X)}$. The Dirac morphism
$(\VV_{\on{O}(X)},E_{\on{O}(X)})\times
(\VV_{\on{O}(X)},E_{\on{O}(X)})\da (\VV_{\on{O}(X)},E_{\on{O}(X)})$
gives rise to a Morita morphism
\[ \pr_1^*\A_{\on{O}(X)}\otimes \pr_2^*\A_{\on{O}(X)}\da \A_{\on{O}(X)},\]
which is associative up to 2-isomorphisms. 

\begin{proposition}\begin{enumerate}\item
  There is a canonical Morita morphism $\C\da \A_{\on{O}(X)}$ with
  underlying map the inclusion of the group unit, $\{I\}\hra
  \on{O}(X)$. \item For any orthogonal decomposition $X=X'\oplus X''$,
  there is a canonical Morita morphism
\[ \pr_1^*\A_{\on{O}(X')}\otimes \pr_2^*\A_{\on{O}(X'')} \da \A_{\on{O}(X)}\]
with underlying map the inclusion $\on{O}(X')\times\on{O}(X'')\hra \on{O}(X)$. 
\end{enumerate}
\end{proposition}
\begin{proof}
The Proposition follows since the restriction of $E_{\on{O}(X)}$ to
$I$ is $X^*$, while the restriction to $\on{O}(X')\times\on{O}(X'')$
is $E_{\on{O}(X')}\times E_{\on{O}(X'')}$. 
\end{proof}
The action of $\on{O}(X)$ by conjugation lifts
to an action on the bundle $\VV_{\on{O}(X)}$, preserving the Dirac
structure $E_{\on{O}(X)}$. Hence $\A_{\on{O}(X)}$ is an
$\on{O}(X)$-equivariant Dixmier-Douady bundle.

The construction of $\A_{\on{O}(X)}$, using the family of boundary
conditions given by orthogonal transformations, is closely related to
a construction given by Atiyah-Segal in \cite{at:twi}, who also
identify the resulting Dixmier-Douady class. The result is most nicely
stated for the restriction to $\SO(X)$; for the general case use an
inclusion $\on{O}(X)\hra \SO(X\oplus \R)$. 
\begin{proposition}  \label{prop:ddclass} \cite[Proposition 5.4]{at:twi}
  Let $(x,y)=\on{DD}(\A_{\SO(X)})$ be the Dixmier-Douady class. 
\begin{enumerate}
\item For $\dim X\ge 3,\ \dim X\not=4$ the class $x$ generates
      $H^3(\SO(X),\Z)=\Z$.
\item For $\dim X\ge 2$ the class $y$ 
  generates $H^1(\SO(X),\Z_2)=\Z_2$.
\end{enumerate}
\end{proposition}
Atiyah-Segal's proof uses an alternative construction $\A_{\SO(X)}$ in
terms of loop groups (see below). Another argument is sketched in
Appendix \ref{app:so2}.

\subsection{Pull-back under exponential map}
Let $(\VV_{\mf{o}(X)},\,E_{\mf{o}(X)})$ be as in Section
\ref{subsec:exp}, and let $\A_{\mf{o}(X)}$ be the resulting
$\on{O}(X)$-equivariant Dixmier-Douady bundle. Since
$E_{\mf{o}}(X)|_a=\on{Gr}_a$, its intersection with $X\subset \XX$ is
trivial, and so $\A_{\mf{o}(X)}$ is Morita trivial.
Recall the Dirac morphism
$(\Pi,-\varpi)\colon (\VV_{\mf{o}(X)},\,E_{\mf{o}(X)})\da
(\VV_{\on{O}(X)},E_{\on{O}(X)})$,
with underlying map $\exp\colon \mf{o}(X)\to \on{O}(X)$.  We
had shown that it is a strong Dirac morphism over the subset
$\mf{o}(X)_\natural$ where the exponential map has maximal rank, or
equivalently where $\Pi_a=(I-e^{-a})/a$ is invertible. One hence obtains a Morita morphism
\[ \A_{\mf{o}(X)}|_{\mf{o}(X)_\natural}\da \A_{\on{O}(X)}.\]
Together with the Morita trivialization $\C\da \A_{\mf{o}(X)}$, this
gives a Morita trivialization of $\exp^*\A_{\on{O}(X)}$ over
$\mf{o}(X)_\natural$.  

On the other hand, $\exp^*E_{\on{O}(X)}$ is the
Lagrangian sub-bundle of $\mf{o}(X)\times\XX$ defined by the map
$a\mapsto \exp(a)\in \on{O}(X)$.  Replacing $\exp(a)$ with $\exp(sa)$,
one obtains a homotopy $E_s$ between $E_1=\exp^*E_{\on{O}(X)}$ and
$E_0=X^*$, hence another Morita trivialization of
$\exp^*\A_{\on{O}(X)}$ (defined over all of $\mf{o}(X)$). Let $L\to
\mf{o}(X)_{\natural}$ be the $\on{O}(X)$-equivariant line bundle
relating these two Morita trivializations.
\begin{proposition}\label{prop:exp2}
Over the component containing $0$, the line bundle $L\to
\mf{o}(X)_\natural$ is $\on{O}(X)$-equivariantly trivial.  In other
words, the two Morita trivializations of
$\exp^*\A_{\on{O}(X)}|_\natural$ are 2-isomorphic over the component
of $\mf{o}(X)_\natural$ containing $0$.
\end{proposition}
\begin{proof}
The linear retraction of $\mf{o}(X)$ onto the origin preserves the 
component of $\mf{o}(X)_\natural$ containing $0$. Hence it suffices to 
show that the $\on{O}(X)$-action on the fiber of $L$ at $0$ is trivial. 
But this is immediate since both Morita trivializations of 
$\exp^*\A_{\on{O}(X)}$ at $0\in \mf{o}(X)_\natural$ coincide 
with the obvious Morita trivialization of $\A_{\on{O}(X)}|_e$. 
\end{proof}

%
\subsection{Construction via loop groups}\label{subsec:loop}
The bundle $\A_{\SO(X)}$ has the following description in terms of
loop groups (cf. \cite{at:twi}). Fix a Sobolev level $s>1/2$, and let
$\P\SO(X)$ denote the Banach manifold of paths $\gamma\colon \R\to \SO(X)$
of Sobolev class $s+1/2$ such that $\pi(\gamma):=\gamma(t+1)\gamma(t)^{-1}$ is
constant. (Recall that for manifolds $Q,P$, the maps $Q\to P$ of
Sobolev class greater than $k+\dim Q/2$ are of class $C^k$.)  The map
\[\pi\colon \P\SO(X)\to \SO(X),\ \gamma\mapsto \pi(\gamma)\] is an
$\SO(X)$-equivariant principal bundle, with structure group
the loop group $L\SO(X)=\pi^{-1}(e)$. Here elements of $\SO(X)$ acts
by multiplication from the left, while loops $\lambda\in L\SO(X)$ acts
by $\gamma\mapsto \gamma\lambda^{-1}$.  Let $\X=L^2([0,1],X)$ carry the complex
structure $J_0$ defined by $\f{\p}{\p t}$ with anti-periodic boundary
conditions, and let $\S_0$ be the resulting spinor module.  The action
of the group $L\SO(X)$ on $\X$ preserves the polarization defined by
$J_0$, and defines a continuous map $L\SO(X)\to
\on{O}_\res(\X)$. Using its composition with the map
$\on{O}_\res(\X)\to \on{PU}(\S_0)$, we have:
\begin{proposition}
  The Dixmier-Douady bundle $\A_{\SO(X)}$ is an associated bundle
  $\P\SO(X)\times_{L\SO(X)}\K(\S_0)$.
\end{proposition}
\begin{proof}
Given $\gamma\in \P\SO(X)$, consider the operator
$M_\gamma$ on $\X=L^2([0,1],X)$ of pointwise multiplication by $\gamma$. 
As in Proposition \ref{prop:hom}, we see that $M_\gamma$ takes 
the boundary conditions $f(1)=-f(0)$ to 
$(M_\gamma f)(1)=-\pi(\gamma) (M_\gamma f)(0)$, and induces an  
isomorphism $\K(\S_0)=\A_I\to \A_{\pi(\gamma)}$.  This defines a map 
\[ \P\SO(X)\times \K(\S_0)\to \A_{\SO(X)}\]
with underlying map $\pi\colon \P\SO(X)\to \SO(X)$. This map is
equivariant relative to the action of $L\SO(X)$, and descends to the
desired bundle isomorphism.
\end{proof}
In particular $\pi^*\A_{\SO(X)}=\P\SO(X)\times \K(\S_0)$ has a Morita
trivialization defined by the trivial bundle $\E_0=\P\SO(X)\times
\S_0$. The Morita trivialization is $\wh{L\SO}(X)\times
\SO(X)$-equivariant, using the central extension of the loop group
obtained by pull-back of the central extension $\on{U}(\S_0)\to
\on{PU}(\S_0)$.

\section{q-Hamiltonian $G$-spaces}\label{sec:qham}
In this Section, we will apply the correspondence between Dirac
structures and Dixmier-Douady bundles to the theory of group-valued
moment maps \cite{al:mom}. Most results will be immediate consequences
of the functoriality properties of this correspondence. Throughout
this Section, $G$ denotes a Lie group, with Lie algebra $\g$.  We
denote by $\xi^L,\xi^R\in\mf{X}(G)$ the left,right invariant vector
fields defined by the Lie algebra element $\xi\in\g$, and by
$\theta^L,\theta^R\in\Om^1(G,\g)$ the Maurer-Cartan forms, defined by
$\iota(\xi^L)\theta^L=\iota(\xi^R)\theta^R=\xi$. For sake of comparison, we
begin with a quick review of ordinary Hamiltonian $G$-spaces from the
Dirac geometry perspective.

\subsection{Hamiltonian $G$-spaces}
A Hamiltonian $G$-space is a triple $(M,\om_0,\Phi_0)$ consisting of a
$G$-manifold $M$, an invariant 2-form $\om_0$ and an equivariant
\emph{moment map} $\Phi_0\colon M\to \g^*$ such that
\begin{enumerate}
\item[(i)] $\d\om_0=0$, 
\item [(ii)] $\iota(\xi_M)\om_0=-\d\l\Phi_0,\xi\r,\ \ \ \xi\in\g$,  
\item[(iii)] $\ker(\om_0)=0$. 
\end{enumerate}
Conditions (ii) and (iii) may be rephrased
in terms of Dirac morphisms. Let $E_{\g^*}\subset \T\g^*$ be the Dirac
structure spanned by the sections 
\[ e_0(\xi)=(\xi^\sharp,\ \l\d\mu,\xi\r),\ \  \xi\in\g. \]
Here $\xi^\sharp\in\mf{X}(\g^*)$ is the vector field generating the
co-adjoint action (i.e. $\xi^\sharp|_\mu=(\ad_\xi)^*\mu$), and
$\l\d\mu,\xi\r\in \Om^1(\g^*)$ denotes the 1-form defined by $\xi$.
Then Conditions (ii), (iii) hold if and only if
\[ (\d\Phi_0,\om_0)\colon (\T M,TM)\da (\T \g^*,E_{\g^*})\]
is a strong Dirac morphism. Using the Morita isomorphism 
$\wt{\Cl}(TM)\da \A_{TM}$, and putting
$\A_{\g^*}^{\on{Spin}}:=\A_{E_{\g^*}}$
we obtain a $G$-equivariant Morita morphism 
\[ (\Phi_0,\E_0)\colon \wt{\Cl}(TM)\da \A_{\g^*}^{\on{Spin}}.\]
Since $E_{\g^*}\cap T\g^*=0$, the zero Dirac morphism
$(T\g^*,E_{\g^*})\da (0,0)$ is strong, hence it defines a Morita
trivialization $\A_{\g^*}^{\on{Spin}}\da \C$. From Proposition
\ref{prop:sympl}, we see that the resulting equivariant
$\Spin_c$-structure $\wt{\Cl}(TM)\da \C$ is 2-isomorphic to the
$\Spin_c$-structure defined by the symplectic form $\om_0$.
(Since symplectic manifolds are even-dimensional, we may
work with $\Cl(TM)$ in place of $\wt{\Cl}(TM)$.)

\subsection{q-Hamiltonian $G$-spaces}
An $\Ad(G)$-invariant inner product $B$ on $\g$ defines a closed
bi-invariant 3-form 
\[ \eta=\f{1}{12}B(\theta^L,[\theta^L,\theta^L])\in\Om^3(G).\]
A q-Hamiltonian $G$-manifold \cite{al:mom} is a $G$-manifold $M$,
together with an an invariant 2-form $\om$, and an
equivariant \emph{moment map} $\Phi\colon M\to G$ such that
\begin{enumerate}
\item[(i)]$\d\om=-\Phi^*\eta$,
\item[(ii)] $\iota(\xi_M)\om=-{\ts \hh}\Phi^* B((\theta^L+\theta^R),\xi)$
\item[(iii)]$\ker(\om)\cap \ker(\d\Phi)=0$ everywhere. 
\end{enumerate}
The simplest examples of q-Hamiltonian $G$-spaces are the conjugacy
classes in $G$, with moment map the inclusion $\Phi\colon \Co\hra G$.
Again, the definition can be re-phrased in terms of Dirac structures. Let $E_G\subset \T G$ be the Lagrangian sub-bundle 
spanned by the sections
\[ e(\xi)=\big(\xi^\sharp,\
{\ts \hh} B(\theta^L+\theta^R,\,\xi)\big),\ \ \xi\in\g.\]
Here $\xi^\sharp=\xi^L-\xi^R\in\mf{X}(G)$ is the vector field
generating the conjugation action. $E_G$ is the
\emph{Cartan-Dirac structure} introduced by Alekseev,
\v{S}evera and Strobl \cite{bur:di,sev:poi}. As shown by Bursztyn-Crainic
\cite{bur:di}, Conditions (ii) and (iii)
  above hold if and only if
\[ (\d\Phi,\om)\colon (\mathbb{T}M, TM)\da (\T G,E_G)\] 
is a strong Dirac morphism. Let 
\[ \A^{\on{Spin}}_G:=\A_{E_G}\]
be the $G$-equivariant Dixmier-Douady bundle over $G$ defined by the
Cartan-Dirac structure. The strong Dirac morphism $(\d\Phi,\om)$
determines a Morita morphism $\A_{TM}\da \A^{\on{Spin}}_G$. Since
$\A_{TM}$ is naturally Morita isomorphic to $\wt{\Cl}(TM)$ we 
obtain a 
distinguished 2-isomorphism class of $G$-equivariant Morita morphisms
\begin{equation}\label{eq:cann} 
(\Phi,\E)\colon \wt{\Cl}(TM)\da \A^{\on{Spin}}_G.
\end{equation}
\begin{definition}
The Morita morphism \eqref{eq:cann} is called the \emph{canonical
  twisted $\Spin_c$-structure} for the q-Hamiltonian $G$-space
$(M,\om,\Phi)$. 
\end{definition}
\begin{remarks}\label{rem:spinc}
\begin{enumerate}
\item  Equation 
  \eqref{eq:cann} generalizes the usual $\Spin_c$-structure for a
  symplectic manifold. Indeed, if $G=\{e\}$ we have
  $\A_G^{\on{Spin}}=\C$, and a q-Hamiltonian $G$-space is just a
  symplectic manifold. Proposition \ref{prop:sympl} shows that the
  composition $\wt{\Cl}(TM)\da \A_{TM}\da \C$ in that case is
  2-isomorphic to the Morita trivialization defined by an
  $\om$-compatible almost complex structure. 
\item The tensor product
  $\wt{\Cl}(TM)\otimes \wt{\Cl}(TM)=\wt{\Cl}(TM\oplus TM)$ is canonically Morita
  trivial (see Section \ref{subsec:clifford}). 
  Hence, the twisted $\Spin_c$-structure on a q-Hamiltonian $G$-space
  defines a $G$-equivariant Morita trivialization 
\begin{equation}\label{eq:canline}
\C\da \Phi^* (\A_G^{\on{Spin}})^{\otimes 2}.
\end{equation} 
One may think of \eqref{eq:canline} as the counterpart to the
canonical line bundle. Indeed, for $G=\{e\}$, \eqref{eq:canline} is a
Morita isomorphism from the trivial bundle over $M$ to itself.  It is
thus given by a Hermitian line bundle, and from (a) above one sees
that this is the \emph{canonical line bundle} associated to the
$\Spin_c$-structure of $(M,\om)$.
\end{enumerate}
\end{remarks}
\begin{remark}\label{rem:ofcourse}
In terms of the trivialization $TG=G\times\g$ given by the
left-invariant vector fields $\xi^L$, the Cartan-Dirac
structure $(\T G,E_G)$ is just the pull-back of the tautological Dirac
structure $(\VV_{\on{O}(\g)},E_{\on{O}(\g)})$ under the adjoint action
$\Ad\colon G\to \on{O}(\g)$. Similarly, $\A_G^{\on{Spin}}$ is simply
the pull-back of $\A_{\on{O}(\g)}\to \on{O}(\g)$ under the map
$\Ad\colon G\to \on{O}(\g)$.
\end{remark}
In many cases q-Hamiltonian $G$-spaces have even dimension, so that we
may use the usual Clifford algebra bundle $\Cl(TM)$ in \eqref{eq:cann}:
\begin{proposition}\label{cor:evenodd}
  Let $(M,\om,\Phi)$ be a connected q-Hamiltonian $G$-manifold. Then
  $\dim M$ is even if and only if $\Ad_{\Phi(m)}\in \SO(\g)$ for all
  $m\in M$. In particular, this is the case if $G$ is connected. 
\end{proposition}
\begin{proof} 
This is proved in \cite{al:du}, but follows much more easily from the
following Dirac-geometric argument. The parity of the Lagrangian
sub-bundle $TM\subset \T M$ is given by $(-1)^{\dim M}=\pm 1$. 
By Proposition \ref{prop:evenif}, the parity is preserved under strong
Dirac morphisms. Hence it coincides with the parity of $E_G$ over
  $\Phi(M)$, and by Remark \ref{rem:ofcourse} this is the same as the parity of the
  tautological Dirac structure $E_{\on{O}(\g)}$ over
  $\Ad(\Phi(M))\subset \on{O}(\g)$. The latter is given by
  $\det(\Ad_\Phi)=\pm 1$. This shows $\det(\Ad_\Phi)=(-1)^{\dim M}$. 
\end{proof}
As a noteworthy special case, we have: 
\begin{corollary}
A conjugacy class $\Co=\Ad(G)g\subset G$ of a compact Lie group $G$ is
even-dimensional if and only if $\det(\Ad_g)=1$.
\end{corollary}

\subsection{Stiefel-Whitney classes}
The existence of a $\Spin_c$-structure on a
symplectic manifold implies the vanishing of the third integral
Stiefel-Whitney class $W^3(M)=\ti{\beta}(w_2(M))$, while of course 
$w_1(M)=0$ by orientability. For q-Hamiltonian spaces we have the
following statement:  
\begin{corollary}
For any q-Hamiltonian $G$-space, 
\[ W^3(M)\equiv \ti{\beta}(w_2(M))=\Phi^* x,\ \ w_1(M)=\Phi^*y.\]
where $(x,y)\in H^3(G,\Z)\times H^1(G,\Z_2)$ is the Dixmier-Douady
class of $\A^{\on{Spin}}_G$. A similar statement holds
for the $G$-equivariant Stiefel-Whitney classes.
\end{corollary}
\begin{remarks}
\begin{enumerate}
\item The result gives in particular a description of $w_1(\Co)$ and
      $\ti{\beta}(w_2(\Co))$ for all conjugacy classes $\Co\subset G$
      of a compact Lie group.
\item If $G$ is simply connected, so that $H^1(G,\Z_2)=0$, it follows that
      $w_1(M)=0$. Hence q-Hamiltonian spaces for simply connected
      groups are orientable. In fact, there is a canonical orientation
      \cite{al:du}. 
\item Suppose $G$ is simple and simply connected. Then $x$ is $\cox$
      times the generator of $H^3(G,\Z)=\Z$, where $\cox$ is the dual
      Coxeter number of $G$.  This follows from Remark
      \ref{rem:ofcourse}, since \[\Ad^*\colon H^3(\SO(\g),\Z)=\Z\to
      H^3(G,\Z)=\Z\] is multiplication by $\cox$. We see that a
      conjugacy class $\Co$ of $G$ admits a $\Spin_c$-structure if and
      only if the pull-back of the generator of $H^3(G,\Z)$ is
      $\cox$-torsion. Examples of conjugacy classes not admitting
      a $\Spin_c$-structure may be found in \cite{me:con}.
\end{enumerate}
\end{remarks}

\subsection{Fusion}
Let $\on{mult}\colon G\times G\to G$ be the group multiplication, and
denote by $\sig\in \Om^2(G\times G)$ the 2-form
\begin{equation}\label{eq:sig2} \sig=-\hh B(\pr_1^*\theta^L,\,\pr_2^*\theta^R)\end{equation}
where $\pr_j\colon G\times G\to G$ are the two projections. By 
\cite[Theorem
3.9]{al:pur} the pair $(\d\on{mult},\sig)$ define a strong
$G$-equivariant Dirac morphism
\[ (\d\on{mult},\sig)\colon (\T G,E_G)\times (\T G,E_G)\da (\T G,E_G).\]
This can also be seen using Remark \ref{rem:ofcourse} and Proposition
\ref{prop:mult}, since left trivialization of $TG$ intertwines
$\d\on{mult}$ with the map $\Sigma$ from \eqref{eq:semidirect}, taking \eqref{eq:sig2} to the 2-form $\sig$ on $V_{\on{O}(\g)}\times
V_{\on{O}(\g)}$. It induces a Morita morphism
\begin{equation}\label{eq:canmult}
 (\on{mult},\E)\colon \pr_1^*\A^{\on{Spin}}_G\otimes\pr_2^*\A^{\on{Spin}}_G\da
 \A^{\on{Spin}}_G. 
\end{equation}
If $(M,\om,\Phi)$ is a q-Hamiltonian $G\times G$-space, 
then $M$ with diagonal $G$-action, 2-form $\om_\fus=\om+\Phi^*\sig$, and
moment map $\Phi_\fus=\on{mult}\circ \Phi\colon M\to G$ defines a 
q-Hamiltonian $G$-space 
\begin{equation}\label{eq:fuss}
 (M,\om_\fus,\Phi_\fus).
\end{equation}
The space \eqref{eq:fuss} is called the \emph{fusion} of
$(M,\om,\Phi)$. Conditions (ii), (iii) hold since
\begin{equation}\label{eq:compsot}
(\d\Phi_\fus,\om_\fus)=(\d\on{mult},\sigma)\circ (\d\Phi,\om)
\end{equation}
is a composition of strong Dirac morphisms, while (i) follows from
$\d\sig=\on{mult}^*\eta-\pr_1^*\eta-\pr_2^*\eta$. The Dirac-Dixmier-Douady
functor (Theorem \ref{th:functor}) shows that the twisted
$\Spin_c$-structures are compatible with fusion, in the following
sense: 
\begin{proposition}
  The Morita morphism $\wt{\Cl}(TM)\da \A^{\on{Spin}}_G$ for the
  q-Hamiltonian $G$-space $(M,\om_\fus,\Phi_\fus)$ is equivariantly
  2-iso\-mor\-phic to the composition of Morita morphisms
\[ \wt{\Cl}(TM)\da \pr_1^*\A^{\on{Spin}}_G\otimes \pr_2^*\A^{\on{Spin}}_G\da
\A^{\on{Spin}}_G\]
defined by the twisted $\Spin_c$-structure for $(M,\om,\Phi)$, followed
by  \eqref{eq:canmult}. 
\end{proposition}

\subsection{Exponentials}\label{sec:exp}
Let $\exp\colon\g\to G$ be the exponential map. The pull-back
$\exp^*\eta$ is equivariantly exact, and admits a canonical primitive
$\varpi\in \Om^2(\g)$ defined by the homotopy operator for the linear
retraction onto the origin. 
\begin{remark}
Explicit calculation shows \cite{al:no}
that $\varpi$ is the pull-back of the 2-form (denoted by the same
letter) $\varpi\in \Gamma(\wedge^2 V_{\mf{o}(\g)}^*)\cong
C^\infty(\mf{o}(\g),\wedge^2\g^*)$ from Section \ref{subsec:exp}
under the adjoint map, $\ad\colon \g\to \mf{o}(\g)$.
Using the inner product to identify $\g^*\cong\g$, the Dirac structure
$E_{\g^*}\equiv E_\g$ is the pull-back of the Dirac structure $E_{\mf{o}(\g)}$ by
the map $\ad\colon \g\to \mf{o}(\g)$. 
\end{remark}
The differential of the exponential map together with the 2-form $\varpi$ define a 
Dirac morphism 
\[ (\d\exp,-\varpi)\colon (\T \g,\,E_\g)\da (\T G,\,E_G) \]
which is a strong Dirac morphism over the open subset $\g_\natural$
where $\exp$ has maximal rank. See \cite[Proposition 3.12]{al:pur},
or Proposition \ref{prop:exp} above.

Let $(M,\Phi_0,\om_0)$ be a Hamiltonian $G$-space with
$\Phi_0(M)\subset \g_\natural$, and $\Phi=\exp\Phi_0,\ 
\om=\om_0-\Phi_0^*\varpi$. Then $(\d\Phi,\om)=(\d\exp,-\varpi)\circ
(\d\Phi_0,\om_0)$ is a strong Dirac morphism, hence $(M,\om,\Phi)$ is
a q-Hamiltonian $G$-space. It is called the \emph{exponential} of the
Hamiltonian $G$-space $(M,\om_0,\Phi_0)$.

The canonical twisted $\Spin_c$-structure for $(M,\om,\Phi)$ can be
composed with the Morita trivialization
$\Phi^*\A^{\on{Spin}}_G=\Phi_0^*\exp^*\A_G^{\on{Spin}}\da \C$ defined
by the Morita trivialization of $\exp^*\A_G^{\on{Spin}}$, to produce
an ordinary equivariant $\Spin_c$-structure. On the other hand, we
have the equivariant $\Spin_c$-structure defined by the symplectic
form $\om_0$.
\begin{proposition} Suppose $(M,\om_0,\Phi_0)$ is a Hamiltonian
  $G$-space, such that $\Phi_0$ takes values in the zero component 
  of $\g_\natural\subset\g$. Let $(M,\om,\Phi)$ be its exponential. Then the
  composition\footnote{We could also write $\Cl(TM)$ in place of
    $\wt{\Cl}(TM)$ since $\dim M$ is even.}
\[ \wt{\Cl}(TM)\da \Phi^*\A^{\on{Spin}}_G\da \C\]
is 2-isomorphic to the Morita morphism $\wt{\Cl}(TM)\da \C$ given by the 
canonical $\Spin_c$-structure for $\om_0$. 
\end{proposition}
\begin{proof}
  Proposition \ref{prop:exp2} shows that over the zero component of 
  $\g_{\natural}$, the Morita trivialization of $\exp^*\A_G^{\on{Spin}}$
  is 2-isomorphic to the composition of the Morita isomorphism
  $\A_\g^{\on{Spin}}\da \A_G^{\on{Spin}}$ induced by
  $(\d\exp,-\varpi)$, with the Morita trivialization of
  $\A_\g^{\on{Spin}}$ (induced by the Dirac morphism $(T\g^*,E_\g)\da
  (0,0)$). The result now follows from Theorem \ref{th:functor}.
\end{proof}

\subsection{Reduction}
In this Section, we will show that the canonical
twisted $\Spin_c$-structure is well-behaved under 
reduction.
Let $(M,\om,\Phi)$ be a q-Hamiltonian $K\times G$-space. 
Thus $\Phi$ has two components $\Phi_K,\Phi_G$, taking values in
$K,G$ respectively.  Suppose 
$e\in G$ a regular value of $\Phi_G$, so that $Z=\Phi_G^{-1}(e)$ is
a smooth $K\times G$-invariant submanifold. Let $\iota\colon Z\to M$
be the inclusion.  The moment map condition 
shows that the $G$-action is locally free on $Z$, and that
$\iota^*\om$ is $G$-basic. Let us assume for simplicity that the
$G$-action on $Z$ is actually free. Then 
\[ M_{\on{red}}=Z/G\]
is a smooth $K$-manifold, the $G$-basic 2-form $\iota^*\om$ descends 
to a 2-form $\om_{\on{red}}$ on $M_{\on{red}}$, and the restriction
$\Phi|_Z$ descends to a smooth $K$-equivariant map
$\Phi_{\on{red}}\colon M_{\on{red}}\to K$. 
\begin{proposition}\cite{al:mom} The triple
  $(M_{\on{red}},\om_{\on{red}},\Phi_{\on{red}})$ is a q-Hamiltonian
  $K$-space. In particular, if $K=\{e\}$ it is a symplectic manifold.
\end{proposition}
We wish to relate the canonical twisted $\Spin_c$-structures for 
$M_\red$ to that for $M$.  We need: 
\begin{lemma}
  There is a $G\times K$-equivariant Morita morphism 
\begin{equation}\label{eq:clifred} 
\wt{\Cl}(TM)|_Z\da  \wt{\Cl}(TM_\red),
\end{equation}
with underlying map the quotient map $\pi\colon Z\to M_\red$.
\end{lemma}
\begin{proof}
Consider the exact sequences of vector bundles over $Z$, 
\begin{equation}\label{eq:exact1}
 0\to Z\times\g\to TZ\to \pi^*TM_\red\to 0,\end{equation}
where the first map is inclusion of the generating vector fields, and 
\begin{equation}\label{eq:exact2}
 0\to TZ\to TM|_Z \to Z\times \g^*\to 0,
\end{equation}
where the map $TM|_Z\to \g^*\cong \g =T_eG$ is the restriction
$(\d\Phi)|_Z$. (We are writing $\g^*$ in \eqref{eq:exact2} to avoid
confusion with the copy of $\g$ in \eqref{eq:exact1}.) The Euclidean
metric on $TM$ gives orthogonal splittings of both exact sequences,
hence it gives a $K\times G$-equivariant direct sum decomposition
\begin{equation}\label{eq:splitting} TM|_Z=\pi^*TM_\red\oplus
Z\times(\g\oplus\g^*).
\end{equation}
The standard symplectic structure 
\begin{equation}\label{eq:omgg} 
\om_{\g\oplus \g^*}((v_1,\mu_1),(v_2,\mu_2))=\mu_1(v_2)-\mu_2(v_1)\end{equation}

defines a $K\times G$-equivariant $\Spin_c$-structure on $Z\times
(\g\oplus\g^*)$, and gives the desired equivariant Morita isomorphism.
\end{proof}
Note that the restriction of the 
Morita morphism $\wt{\Cl}(TM)\da \A_{K\times G}^{\on{Spin}}$ 
to $Z\subset M$ takes values in $\A_{K\times
  G}^{\on{Spin}}|_{K\times\{e\}}$. Let 
\begin{equation}\label{eq:restrict}
 \A_{K\times  G}^{\on{Spin}}|_{K\times\{e\}}\da \A_K^{\on{Spin}}\end{equation}
be the Morita isomorphism defined by the Morita trivialization of 
$\A_G^{\on{Spin}}|_{\{e\}}$. The twisted $\Spin_c$-structure for
$(M,\om,\Phi)$ descends to the twisted $\Spin_c$-structure for 
the $G$-reduced space $(M_\red,\om_\red,\Phi_\red)$, in the following 
sense.

\begin{theorem}[Reduction]\label{th:reduction}
Suppose $(M,\om,\Phi)$ is a q-Hamiltonian $K\times G$-manifold, such
that $e$ is a regular value of $\Phi_G$ and such that $G$ acts freely
on $\Phi_G^{-1}(e)$.
The diagram of $K\times G$-equivariant Morita morphisms
\[ \xymatrix{
 {\wt{\Cl}(TM)|_Z}\ar@{-->}[r]\ar @{-->}[d]  & 
**[r]{\A_{K\times G}^{\on{Spin}}|_{K\times\{e\}}}\ar@{-->}[d]\\
{\wt{\Cl}(TM_\red)}\ar@{-->}[r] & {\A_K^{\on{Spin}}}
}\]
%
commutes up to equivariant 2-isomorphism. Here the vertical maps are
given by \eqref{eq:clifred} and \eqref{eq:restrict}.
\end{theorem}

The proof uses the following normal form result for $TM|_Z$. 
\begin{lemma}\label{lem:decomp}
  For a suitable choice of invariant Euclidean metric on $TM$, the
  decomposition $TM|_Z=\pi^*TM_\red\oplus
  Z\times(\g\oplus\g^*)$ from \eqref{eq:splitting} is compatible with
  the 2-forms. That is, 
\[ \om|_Z=\pi^*\om_\red\oplus \om_{\g\oplus\g^*}.\]
\end{lemma}
\begin{proof}
  We will construct $K\times G$-equivariant splittings of the exact
  sequences \eqref{eq:exact1} and \eqref{eq:exact2} so that
  \eqref{eq:splitting} is compatible with the 2-forms. (One may then
  take an invariant Euclidean metric on $TM|_Z$ for which these
  splittings are orthogonal, and extend to $TM$.) Begin with an
  arbitrary $K\times G$-invariant splitting
\[ TM|_Z=TZ\oplus F.\]
Since $F\cap \ker(\om)=0$, the sub-bundle $F^\om\subset TM|_Z$ (the
set of vectors $\om$-orthogonal to all vectors in $F$) has codimension
$\on{codim}(F^\om)=\dim F=\dim\g$.  The moment map condition shows
that $\om$ is non-degenerate on $F\oplus Z\times\g$. Hence
$(Z\times\g)\cap F^\om=0$, and therefore 
\[ TM|_Z=(Z\times\g)\oplus F^\om.\]
Let $\phi\colon TM|_Z\to Z\times \g$ be the projection along $F^\om$. The subspace
\[ F'=\{v-{\ts\hh} \phi(v)|\ v\in F\}\]
is again an invariant complement to $TZ$ in $TM|_Z$, and it is
\emph{isotropic} for $\om$. Indeed, if $v_1,v_2\in F$, 
\[ \om(v_1-{\ts\hh} \phi(v_1),\,v_2-{\ts\hh} \phi(v_2))
={\ts{\hh}}\om(v_1,v_2-\phi(v_2))+{\ts \hh}\om(v_1-\phi(v_1),v_2)\] 
vanishes since $v_i-\phi(v_i)\in F^\om$. 
The restriction of $(\d\Phi_G)|_Z\colon TM|_Z\to \g^*$ to $F'$
identifies $F'=Z\times\g^*$.  We have
hence shown the existence of an invariant decomposition
$TM|_Z=TZ\oplus Z\times\g^*$ where the second summand is embedded as an
$\om$-isotropic subspace, and such that $(\d\Phi_G)|_Z$ is projection to
the second summand. From the $G$-moment map condition
\[\iota(\xi_M)\om|_Z=-{\ts \hh}\Phi_G^*B((\theta^L+\theta^R)|_Z,
\xi)=-B((\d\Phi_G)|_Z,\xi),\ \ \xi\in\g,\]
we see that the induced 2-form on the sub-bundle $Z\times(\g\oplus\g^*)$ is
just the standard one, $\om_{\g\oplus\g^*}$.  The $\om$-orthogonal space
$Z\times(\g\oplus\g^*)^\om$ defines a complement to $Z\times\g\subset
TZ$, and is hence identified with $\pi^*TM_\red$.
\end{proof}

\begin{proof}[Proof of Theorem \ref{th:reduction}]
Let $\Theta\colon TM|_Z\da TM_\red$ be the bundle morphism given by
projection to the first summand in  \eqref{eq:splitting}, followed by
the quotient map. Then 
\[ (\Theta,\om_{\g\oplus\g^*})\colon (\T M|_Z,  TM|_Z)\da (\T M_\red, TM_\red),\]
is a strong Dirac morphism, and the resulting Morita morphism 
$\A_{TM}|_Z\da \A_{TM_\red}$ fits into a commutative diagram 
\begin{equation}\label{eq:dia1} \xymatrix{
{\wt{\Cl}(TM)|_Z} 
\ar@{-->}[r]\ar @{-->}[d]  
& {\A_{TM}|_Z} \ar @{-->}[d] \\
{\wt{\Cl}(TM_\red)}  \ar@{-->}[r] & {\A_{TM_\red}}
}\end{equation}
%
%
On the other hand, letting $\pr_1\colon T(K\times G)|_{K\times\{e\}}\to TK$ be projection
to the first summand, we have 
\[ (\pr_1,0)\circ (\d\Phi|_Z,\om|_Z)=(\d\Phi_\red,\om_\red)\circ
(\Theta,\om_{\g\oplus\g^*}),\]
so that the resulting diagram of Morita morphisms
\begin{equation}\label{eq:dia2} \xymatrix{
{\A_{TM}|_Z} \ar@{-->}[r]\ar @{-->}[d]  &{\A_{K\times G}^{\on{Spin}}|_{K\times\{e\}}}\ar@{-->}[d]\\
\A_{TM_\red}  \ar@{-->}[r] & \A_K^{\on{Spin}}
}\end{equation}
commutes up to 2-isomorphism.  Placing \eqref{eq:dia1} next to
\eqref{eq:dia2}, the Theorem follows.
\end{proof}

\begin{remark}
  If $e$ is a regular value of $\Phi_G$, but the action of $G$ on $Z$
  is not free, the reduced space $M_\red$ is usually an orbifold. The
  Theorem extends to this situation with obvious modifications.
\end{remark}

\begin{remark}
  Reduction at more general values $g\in G$ may be expressed in terms of
  reduction at $e$, using the \emph{shifting trick}: Let $G_g\subset
  G$ be the centralizer of $g$, and $\Ad(G)g^{-1}\cong G/G_{g^{-1}}$ the
  conjugacy class of $g^{-1}$. Then
\[ M\qu_g G:=\Phi_G^{-1}(g)/G_g=(M\times \Ad(G)g^{-1})\qu G\]
where $M\times \Ad(G).g^{-1}$ is the fusion product. Again, one finds that 
$g$ is a regular value of $\Phi_G$ if and only if the $G_g$-action on 
$\Phinv(g)$ is locally free, and if the action is free then $M\qu_g G$
is a q-Hamiltonian $K$-space. 
\end{remark}

\section{Hamiltonian $LG$-spaces}\label{sec:loop}
In his 1988 paper, Freed \cite{fr:lg} argued that for a compact,
simple and simply connected Lie group $G$, the canonical line bundle over the K\"ahler
manifold $LG/G$ (and over the other coadjoint orbits of the loop
group) is a $\wh{LG}$-equivariant Hermitian line bundle $K\to LG/G$,
where the central circle of $\wh{LG}$ acts with a weight $-2\cox$,
where $\cox$ is the dual Coxeter number. In \cite{me:can}, this was
extended to more general Hamiltonian $LG$-spaces. 

In this Section we will use the correspondence between Hamiltonian
$LG$-spaces and q-Hamiltonian $G$-spaces to give a new construction of
the canonical line bundle, in which it is no longer necessary to
assume $G$ simply connected. We begin by recalling the definition of a
Hamiltonian $LG$-space. Let $G$ be a compact Lie group, with a given
invariant inner product $B$ on its Lie algebra. We fix $s>1/2$, and
take take the loop group $LG$ to be the Banach Lie group of maps
$S^1\to G$ of Sobolev class $s+1/2$. Its Lie algebra $L\g$ consists of
maps $S^1\to \g$ of Sobolev class $s+1/2$. We denote by $L\g^*$ the
$\g$-valued 1-forms on $S^1$ of Sobolev class $s-1/2$, with the gauge
action $g\cdot\mu=\Ad_g(\mu)-g^*\theta^R$. A \emph{Hamiltonian
  $LG$-manifold} is a Banach manifold $N$ with an action of $LG$, an
invariant (weakly) symplectic 2-form $\sig\in \Om^2(N)$, and a smooth
$LG$-equivariant map $\Psi\colon N\to L\g^*$ satisfying the moment map
condition
\[ \iota(\xi^\sharp)\sig=-\d \l\Psi,\xi\r,\ \ \ \xi\in L\g.\]
Here the pairing between elements of $L\g^*$ and of $L\g$ is given by
the inner product $B$ followed by integration over $S^1$. 

Suppose now that $G$ is connected, and let $\P G$ be the space of
paths $\gamma\colon \R\to G$ of Sobolev class $s+1/2$ such that
$\pi(\gamma)=\gamma(t+1)\gamma(t)^{-1}$ is constant. The map
$\pi\colon \P G \to G$ taking $\gamma$ to this constant is a
$G$-equivariant principal $LG$-bundle, where $a\in G$ acts by
$\gamma\mapsto a\gamma$ and $\lambda\in LG$ acts by $\gamma\mapsto
\gamma\lambda^{-1}$. One has $\P G/G\cong L\g^*$ with quotient map
$\gamma\mapsto \gamma^{-1}\dot{\gamma}\d t$. Let $\wt{N}\to N$ be the
principal $G$-bundle obtained by pull-back of the bundle $\P G\to
L\g^*$, and $\wt{\Psi}\colon \wt{N}\to \P G$ the lifted moment map.
Then $\wt{\Psi}$ is $LG\times G$-equivariant. Since the $LG$-action on
$\P G$ is a principal action, the same is true for the action on
$\wt{N}$. Assuming that $\Psi$ (hence $\wt{\Psi}$) is \emph{proper},
one obtains a smooth \emph{compact} manifold $M=\ti{N}/LG$ with an induced
$G$-map $\Phi\colon M\to G=\P G/LG$.
\[\begin{CD}
\wt{N} @>{\wt{\Psi}}>> \P G\\
@V{\pi_M}VV @VV{\pi_G}V\\
M @>>{\Phi}> G
\end{CD}\]
In \cite{al:mom}, it was shown how 
to obtain an invariant 2-form $\om$ on $M$, making $(M,\om,\Phi)$ into
a q-Hamiltonian $G$-spaces. This construction sets up a 1-1
correspondence between Hamiltonian $LG$-spaces with proper moment maps
and q-Hamiltonian spaces. 

As noted in Remark \ref{rem:spinc}, the canonical twisted
$\Spin_c$-structure for $(M,\om,\Phi)$ defines a $G$-equivariant
Morita trivialization of the bundle $\E\colon \C\da
\Phi^*\A_G^{\on{Spin}^{\otimes 2}}$ over $M$. 
On the other hand, let $\wh{LG}^{\on{Spin}}$ be the pull-back of the
basic central extension $\wh{L\SO(\g)}$ under the adjoint action. 
By the discussion in Section \ref{subsec:loop}, the
pull-back bundle $\A_G^{\on{Spin}}$ to $\P G$ has a canonical
$\wh{LG}^{\on{Spin}}\!\!\times G$-equivariant Morita trivialization,
\[ \S_0\colon \C\da \pi_G^* \A_G^{\on{Spin}},\]
where the central circle of $\wh{LG}^{\on{Spin}}$ acts with weight $1$.
Tensoring $\S_0$ with itself, and pulling everything back to $\wh{N}$
we obtain two Morita trivializations $\pi_M^*\E$ and
$\wt{\Psi}^*(\S_0\otimes\S_0)$ of the Dixmier-Douady bundle
$\mathcal{C}$ over $\wt{N}$, given by the pull-back of
$\A_G^{\on{Spin}^{\otimes 2}}$ under $\Phi\circ \pi_M=\pi_G\circ
\wt{\Psi}$. Let
\[ \wt{K}:=\on{Hom}_{\mathcal{C}}(\wt{\Psi}^*(\S_0\otimes\S_0),
\pi_M^*\E)\]
Then $\wt{K}$ is a $\wh{LG}^{\on{Spin}}\!\!\times G$-equivariant
Hermitian line bundle, where the central circle in
$\wh{LG}^{\on{Spin}}$ acts with weight $-2$. Its quotient $K=
\wt{K}/G$ is the desired \emph{canonical bundle} for the
Hamiltonian $LG$-manifold $N$.
\begin{remark}
For $G$ simple and simply connected, the central extension
$\wh{LG}^{\on{Spin}}$ is the $\cox$-th power of the `basic central'
extension $\wh{LG}$. We may thus also think of $K_N$ as a
$\wh{LG}$-equivariant line bundle where the central circle acts with
weight $-2\cox$. 
\end{remark}
The canonical line bundle is well-behaved under symplectic reduction.
That is, if $e$ is a regular value of $\Phi$ then $0\in L\g^*$ is also
a regular value of $\Psi$, and $\Phi^{-1}(e)\cong \Psi^{-1}(0)$ as
$G$-spaces.  Assume that $G$ acts freely on these level sets, so that
$M\qu G=N\qu G$ is a symplectic manifold. The canonical line bundle
for $M\qu G$ is simply $K_{M\qu G}=K_N|_{\Psi^{-1}(0)}/G$. As in
\cite{me:can}, one can sometimes use this fact to compute the canonical
line bundle over moduli spaces of flat $G$-bundles over surfaces.

\begin{appendix}
\section{Boundary conditions}\label{subsec:boundary}
In this Section, we will prove several facts about the operator
$\f{\p}{\p t}$ on the complex Hilbert-space $L^2([0,1],\C^n)$,
with boundary conditions defined by $A\in \U(n)$,
\[\on{dom}(D_A)=\{f\in L^2([0,1],\C^n)|\ \dot{f}\in L^2([0,1],\C^n),\ f(1)=-Af(0)\}.\] 
Let $e^{\tpi\lambda^{(1)}},\ldots,e^{\tpi\lambda^{(n)}}$ be the
eigenvalues of $A$, with corresponding normalized eigenvectors
$v^{(1)},\ldots,v^{(n)}\in\C^n$.  Then the spectrum of $D_A$ is given
by the eigenvalues
$ \tpi (\lambda^{(r)}+ k-\hh),\ \ k\in\Z,\ r=1,\ldots,n$
with eigenfunctions 
\[ \phi_{k}^{(r)}(t)=\exp(\tpi \,  (\lambda^{(r)}+ k-\hh)\,t)\,v^{(r)}.
\]
We define $J_A=\i\sign(-\i D_A)$; this coincides with $J_A=D_A/|D_A| $
if $D_A$ has trivial kernel.
\begin{proposition}\label{prop:p1}
  Let $A,A'\in \on{U}(n)$. Then $J_{A'}-J_A$ is
  Hilbert-Schmidt if and only if $A'=A$.
\end{proposition}
\begin{proof}
Suppose $A'\not=A$. Let $\Pi,\Pi'$ be the orthogonal projection operators onto
$\ker(J_A-\i),\ \ker(J_{A'}-\i)$. It suffices to show that $\Pi'-\Pi$
is not Hilbert-Schmidt, i.e. that $(\Pi'-\Pi)^2$ is not trace class.
Since
\[(\Pi-\Pi')^2=\Pi(I-\Pi')\Pi+(I-\Pi)\Pi'(I-\Pi).\]
is a sum of two positive operators, it suffices to show that
$\Pi(I-\Pi')\Pi$ is not trace class. 
%
%
Let $\phi_l^{'(s)}$ be the eigenfunctions of $D_{A'}$, defined similar
to those for $D_A$, with eigenvalues $\tpi (\lambda^{'(s)}+ l-\hh)$. 
Indicating the eigenvalues and eigenfunctions for $A'$ by a prime $'$,
we have
\[ \on{tr}(\Pi(I-\Pi')\Pi)=\sum \Big|\l\phi_{k}^{(r)},\phi_{l}^{'(s)}\r\Big|^2.\]
 where the sum is over all $k,r,l,s$ satisfying
$\lambda^{(r)}+k-\hh>0$ and $\lambda^{'(s)}+l-\hh\le 0$. But
\[ \Big|\l\phi_{k}^{(r)},\phi_{l}^{'(s)}\r\Big|^2=
\Big|\f{\l v^{(r)},v^{'(s)}\r
  \,\,(e^{\tpi(\lambda^{'(s)}-\lambda^{(r)})}-1)}{2\pi (\lambda^{'(s)}-\lambda^{(r)}+l-k)}\Big|^2.\]
Since $A'\not=A$, we can choose $r,s$ such that
\[ e^{\tpi\lambda^{(r)}}\not=e^{\tpi\lambda^{'(s)}}\ \ \mbox{ and }\ \ \l
v^{(r)},v^{'(s)}\r\not=0.\] For such $r,s$, the enumerator is a
non-zero constant, and the sum over $k,l$ is divergent.
\end{proof}

\begin{proposition}\label{prop:p2}
  Given $A,A'\in \U(n)$, let
\[ \gamma\colon [0,1]\to \on{Mat}_n(\C)\]
be a continuous map with 
\[A'\gamma(0)=\gamma(1)A,\] 
and such that
$\dot{\gamma}\in L^\infty([0,1],\on{Mat}_n(\C))$.  Let $M_\gamma$ be
the bounded operator on $L^2([0,1],\C^n)$ given as multiplication by
$\gamma$. Then
\[ M_\gamma J_A -J_{A'}M_\gamma\] is
Hilbert-Schmidt.
\end{proposition}
\begin{proof}
  This is a mild extension of Proposition(6.3.1) in Pressley-Segal
  \cite[page 82]{pr:lo}, and we will follow their line of argument.
  Using the notation from the proof of Proposition \ref{prop:p1}, it
  suffices to show that $M_\gamma \Pi -\Pi' M_\gamma$ is
  Hilbert-Schmidt, or equivalently that both $(I-\Pi')M_\gamma \Pi$ and
  $\Pi' M_\gamma (I-\Pi)$ are Hilbert-Schmidt. We will give the
  argument for $\Pi' M_\gamma (I-\Pi)$, the discussion for $
  (I-\Pi')M_\gamma \Pi$ is similar. We must prove that
\[ \begin{split}
\on{tr}((\Pi' M_\gamma (I-\Pi))(\Pi' M_\gamma (I-\Pi))^*)&=
\on{tr}(\Pi' M_\gamma (I-\Pi) M_\gamma^*)\\& =\sum |\l
\phi_k^{'(r)}|\,M_\gamma|\,\phi_l^{(s)}\r|^2<\infty,\end{split}\] 
where the sum is over all $k,r$ with $\lambda^{'(r)}+k-\hh > 0$ and
over all $l,s$ with $\lambda^{(s)}+l-\hh \le 0$. Changing the sum by
only finitely many terms, we may replace this with a summation over
all $k,r,l,s$ such that $k>0$ and $l\le 0$. Since
$\l\phi_k^{'(r)}|M_{{\gamma}}|\phi_l^{(s)}\r=\l\phi_{k+n}^{'(r)}|M_{{\gamma}}|\phi_{l+n}^{(s)}\r$
for all $n\in \Z$, and since there are $m$ terms with fixed
$k-l=m$, the assertion is equivalent to
\begin{equation}\label{eq:est1} \sum_{r,s}\sum_{m>0} m\ 
|\l\phi_0^{'(r)}|M_{{\gamma}}|\phi_m^{(s)}\r|^2<\infty.\end{equation}
To obtain this estimate, we use $\dot{\gamma}\in
L^\infty([0,1],\on{Mat}_n(\C))$. We have
\[\sum_{r,s}\sum_{m\in\Z}
|\l \phi_0^{'(r)}|M_{\dot{\gamma}}|\phi_m^{(s)}\r|^2
=\sum_r ||M_{\dot{\gamma}}^*\phi_0^{'(r)}||^2<\infty.\]
An
integration by parts shows
\[ \begin{split}
\l\phi_0^{'(r)}|M_{\dot{\gamma}}|\phi_m^{(s)}\r=&-\tpi(\lambda^{(s)}-\lambda^{'(r)}+m)
\l\phi_0^{'(r)}|M_{{\gamma}}|\phi_m^{(s)}\r\\
&+ \l\phi_0^{'(r)}(1)|\gamma(1)| \phi_m^{(s)}(1)\r
- \l\phi_0^{'(r)}(0)|\gamma(0)|\phi_m^{(s)}(0)\r.
\end{split}\]
The boundary terms cancel since $A'\gamma(0)=\gamma(1) A$, and
\[\phi_0^{'(r)}(1)=-A'\phi_0^{'(r)}(0),\ \ \phi_m^{(s)}(1)=-A \phi_m^{(s)}(0).\]%
Hence we obtain
\[\sum_{r,s}\sum_{m\in\Z} (\lambda^{(s)}-\lambda^{'(r)}+ m)^2\   |\l
\phi_0^{'(r)}|M_\gamma|\phi_m^{(s)}\r|^2 <\infty \]
which implies \eqref{eq:est1}. 
\end{proof}

\begin{proposition}\label{prop:p3}
  Let $A\in \U(n)$, and let $\mu\in L^\infty([0,1],\mf{u}(n))$.
  Consider $D_{A,\mu}=D_A+M_\mu$ with domain equal to that of $D_A$,
  and define $J_{A,\mu}$ similar to $J_A$.  Then $J_{A,\mu}-J_A$ is
  Hilbert-Schmidt.
\end{proposition}
\begin{proof}
  Let $\gamma\in C([0,1],\on{U}(n))$ be the solution of the initial
  value problem $\dot{\gamma}\gamma^{-1}=-\mu$ with $\gamma(0)=I$. Let
  $A=\gamma(1) A'$. The operator $M_\gamma$ of multiplication by
  $\gamma$ takes $\on{dom}(D_{A'})$ to $\on{dom}(D_{A})$, and
\[ M_\gamma D_{A'}
M_\gamma^{-1}=D_{A}-\dot{\gamma}\gamma^{-1}=D_{A,\mu}.\]
Hence $M_\gamma J_{A'} M_{\gamma^{-1}}=J_{A,\mu}$. By Proposition
\ref{prop:p2}, $M_\gamma J_{A'} M_{\gamma^{-1}}$
differs from $J_A$ by a Hilbert-Schmidt operator.
\end{proof}

Let us finally consider the continuity properties of the family of
operators $D_A,\ A\in \on{U}(n)$. Recall \cite[Chapter VIII]{ree:fu}
that the \emph{norm resolvent topology} on the set of unbounded
skewadjoint operators on a Hilbert space is defined by declaring that
a net $D_i$ converges to $D$ if and only if $R_1(D_i)=(D_i-I)^{-1}\to
R_1(D)=(D-I)^{-1}$ in norm.  This then implies that $R_z(D_i)\to
R_z(D)$ in norm, for any $z$ with non-zero real part, and in fact
$f(D_i)\to f(D)$ in norm for any bounded continuous function $f$. For
bounded operators, convergence in the norm resolvent topology is
equivalent to convergence in the norm topology. 

\begin{proposition}\label{prop:resolv}
The map $A\mapsto D_A$ is continuous in the norm resolvent topology.
\end{proposition}
\begin{proof}
  We will use that $||R_1(D)||=||(D-I)^{-1}||<1$ for any skew-adjoint operator
  $D$. Let us check continuity at any given $A\in \on{U}(n)$.
  Given $a\in \mf{u}(n)$, let us write $D_a=D_{\exp(a)A}$. We will
  prove continuity at $A$ by showing that
\[ ||R_1(D_a)-R_1(D_0)||\le 3||a||.\]
Let $U_a\in \on{U}(\V)$ be the operator of
pointwise multiplication by $\exp(ta)\in \on{U}(V)$. Then
\[ ||U_a-U_0||=\on{sup}_{t\in[0,1]}||\exp(ta)-I||\le ||a||.\] 
The operator $U_a$ takes the domain of $D_0$ to that of
$D_a$, since $f(1)=-A f(0)$ implies $(U_a f)(1)=
\exp(a)f(1)=-\exp(a)A f(0)$. Furthermore, 
\[ D_a=U_a (D_0+M_a) U_a^{-1}\]
Hence 
\[ R_1(D_a)=U_a\  R_1(D_0+M_a) \ U_a^{-1}.\]
The second resolvent identity
$R_1(D_0+M_a)-R_1(D_0)=R_1(D_0+M_a) M_a R_1(D_0)$ shows
\[ ||R_1(D_0+M_a)-R_1(D_0)||\le ||M_a||=||a|| .\]
Hence
\[ \begin{split} \lefteqn{||R_1(D_a)-R_1(D_0)|| = ||U_a R_1(D_0+M_a) U_a^{-1}
  - U_0 R_1(D_0) U_0^{-1}|| }\\
&\le  ||(U_a-U_0)R_1(D_0+M_a) U_a^{-1}||+||U_0 R_1(D_0+M_a)
(U_a^{-1}-U_0^{-1})||\\ &
\ \ +||U_0 (R_1(D_0+M_a)-R_1(D_0))U_0^{-1}|| \\
&\le 2 ||a||\ ||R_1(D_0+M_a)|| +||R_1(D_0+M_a)-R_1(D_0)||< 3 ||a||. \qedhere
\end{split}
\]
\end{proof}

\section{The Dixmier-Douady bundle over $S^1$}\label{app:so2}
Let $S^1=\R/\Z$ carry the \emph{trivial} action of $S^1$. The Morita
isomorphism classes of $S^1$-equivariant Dixmier-Douady bundles $\A\to
S^1$ are labeled by their class
\[ \on{DD}_{S^1}(\A)\in H^3_{S^1}(S^1,\Z)
\times H^1(S^1,\Z_2).\]
The bundle corresponding to $x\in H^3_{S^1}(S^1,\Z)=H^2_{S^1}(\pt)=\Z$
and $y\in H^1(S^1,\Z_2)=H^0(\pt,\Z_2)=\Z_2$ may be described as
follows. Let $L_{(x,y)}\cong\C$ be the $\Z_2$-graded
$S^1$-representation, of parity given by the parity of $y$, and with
$S^1$-weight given by $x$.  Choose a $\Z_2$-graded $S^1$-equivariant
Hilbert space $\H$ with an equivariant isomorphism $\tau\colon \H\to
\H\otimes L$ preserving $\Z_2$-gradings. Then $\tau$ induces an
$S^1$-equivariant $*$-homomorphism 
\[ \ol{\tau}\colon \K(\H)\to \K(\H\otimes L)=\K(\H),\] 
preserving $\Z_2$-gradings. The bundle $\A\to S^1$ with Dixmier-Douady
class $(x,y)$ is obtained from the trivial bundle $[0,1]\times
\K(\H)$, using $\ol{\tau}$ to glue $\{0\}\times \K(\H)$ and $\{1\}\times
\K(\H)$. Given another choice $\H',\tau'$, one obtains a Morita
isomorphism $\E\colon \A\to \A'$, where $\E$ is obtained from a 
similar boundary identification for $[0,1]\times \K(\H',\H)$.

A convenient choice of $H,\tau$ defining the bundle with $x=1,y=1$ is
as follows.  Let $\H$ be a Hilbert space with orthonormal basis of the
form $s_K$, indexed by the subsets $K=\{k_1,k_2,\ldots\}\subset \Z$ 
such that $k_1>k_2>\cdots$ and $k_l=k_{l+1}+1$ for $l$ sufficiently
large. Let
\[ m_K=\# \{k\in K|k>0\}-\# \{k\in \Z-K|\ k\le 0\}.\]
Let $\H$ carry the $S^1$-action such that $s_K$ is a weight vector of
weight $m_K$, and a $\Z_2$-grading, defined by the weight spaces of
even/odd weight.  Let $\tau(K)=\{k+1|\ k\in K\}$.  Then
$m_{\tau(K)}=m_K+1$, hence the automorphism $\tau\colon \H\to \H$
taking $s_K$ to $s_{\tau(K)}$ has the desired properties. 

The Hilbert space $\H$ can also be viewed as a spinor module. Let $\V$
be a real Hilbert space, with complexification $\V^\C$, and let $f_k,\ 
k\in\Z$ be vectors such that $f_k$ together with $f_k^*$ are an
orthonormal basis. The elements $s_K$ for $K=\{k_1,k_2,\cdots\}$ with
$k_1>k_2>\cdots$ are written as formal infinite wedge products
\[ s_K=f_{k_1}\wedge f_{k_2}\wedge \cdots\]
suggesting the action of the Clifford algebra: $\varrho(f_k)$ acts by
exterior multiplication, while $\varrho(f_{k^*})$ acts by
contraction. The automorphism $\tau\in \U(\H)$ is an
implementer of the orthogonal transformation $T\in \on{O}(V)$, 
\begin{equation} \label{eq:t}
Tf_k=f_{k+1},\ \ Tf_k^*=f_{k+1}^*.\end{equation} 
Let us denote the resulting Dixmier-Douady bundle by $\A_{(1,1)}$. 

\begin{proposition}
The Dixmier-Douady bundle $\A_{(1,1)}\to S^1$ is equivariantly
isomorphic to the Dixmier-Douady bundle  $\A\to \SO(2)\cong S^1$, 
constructed as in Section
\ref{subsec:orth}.
\end{proposition}
\begin{proof}
For $s\in \R$, let $A_s\in \SO(2)$ be the matrix of rotation by
$2\pi s$, and let $D_s$ be the skew-adjoint operator $\f{\p}{\p t}$ on
$L^2([0,1],\R^2)$ with boundary conditions $f(1)=-A_s f(0)$.  The
operator $D_0$ has an orthonormal system of eigenvectors $f_k,f_k^*,\ 
k\in\Z$ given by
\[  f_k(t)=e^{\tpi(k-\hh)t}u,\]
with $u=\f{1}{\sqrt{2}}(1,\i)$. The eigenvalues for $f_k,\,f_k^*$ are
$\pm \tpi(k-\hh)$. We see that the $+\i$ eigenspace of $J=D_0/|D_0|$
is given by 
\[ \V_+=\on{span}\{\cdots,f_3,f_2,f_1,f_0^*,f_{-1}^*,\cdots\}.\]
There is a unique isomorphism of $\Cl(\V)$-modules $\S_J\to \H$ taking
the `vacuum vector' $1\in \S_J=\ol{\wedge\V_+}$ to the `vacuum vector'
$f_0\wedge f_{-1}\wedge \cdots$.

For $s\in \R$, define orthogonal transformations $U_s\in \on{O}(\V)$,
where $U_s$ is pointwise multiplication by $t\mapsto A_{st}$. On $f_k$
the operator $U_s$ acts as multiplication by $e^{\tpi st}$, and on
$f_k^*$ as multiplication by $e^{-\tpi st}$. Hence
\[ f_k^{(s)}=U_s f_k,\ \ \ (f_k^{(s)})^*=U_s f_k^*\]
are the eigenvectors of $D_s$, with shifted eigenvalues
$\pm \tpi(k-\hh+s)$. The complex structure
\[ J_s=U_s J U_s^{-1}\]
differs from $J_{D_s}=\i\sign(-\i D_s)$ by a finite rank operator.
Hence, letting $\S_s$ denote the $\Cl(\V)$-module defined by $J_s$,
the fiber of $\A\to \SO(2)$ at $A(s)$ may be described as $\K(\S_s)$.
The orthogonal transformation $U_s$ extends to an orthogonal
transformation of $\ol{\wedge\V}$, taking $\S=\ol{\wedge\V_+}$ to
$\ca{S}_s=\ol{\wedge\V_{+,s}}$, where $\V_{\pm,s}=U_s \V_\pm$. Hence
each $\S_s$ is identified with $\S\cong \H$ as a Hilbert space (not as
a $\Cl(\V)$-module).  The identification $\K(\S_0)\cong \K(\S_1)$ is
given by the choice of any isomorphism of $\Cl(\V)$-modules $\S_0\to
\S_1$.  In terms of the identifications with $\H$, such an isomorphism
is given by an implementer of the orthogonal transformation $U_1$. The
proof is completed by the observation that $U_1=T$ (cf. \eqref{eq:t}),
which is implemented by $\tau$.
\end{proof}

We are now in position to outline an alternative argument for the
computation of the Dixmier-Douady class of $\A_{\SO(n)}$, Proposition
\ref{prop:ddclass}.  Note that $\A_{\SO(n)}$ is equivariant under the
conjugation action of $\SO(n)$.  One has $H^3_{\SO(n)}(\SO(n),\Z)=\Z$
for $n\ge 2,\ n\not=4$, and the natural maps to ordinary cohomology are
isomorphisms for $n\ge 3,\,n\not=4$. Similarly $H^1_{\SO(n)}(\SO(n),\Z_2)=\Z_2$
for $n\ge 2$, and the natural map to $H^1(\SO(n),\Z_2)$ is an
isomorphism. On the other hand, the map
$H^3_{\SO(n)}(\SO(n),\Z)\to H^3_{\SO(2)}(\SO(2),\Z)$ 
(defined by the inclusion $\SO(2)\hra \SO(n)$ as the upper left corner)
is an isomorphism for $n\ge 2,\ n\not=4$, and likewise for 
$H^1(\cdot,\Z_2)$.  It
hence suffices to check that the bundle over $\SO(2)$ has
\emph{equivariant} Dixmier-Douady class $(1,1)\in \Z\times\Z_2$. But
this is clear from our very explicit description of $\A_{\SO(2)}$.
\end{appendix}

\def\cprime{$'$} \def\polhk#1{\setbox0=\hbox{#1}{\ooalign{\hidewidth
  \lower1.5ex\hbox{`}\hidewidth\crcr\unhbox0}}} \def\cprime{$'$}
  \def\cprime{$'$} \def\polhk#1{\setbox0=\hbox{#1}{\ooalign{\hidewidth
  \lower1.5ex\hbox{`}\hidewidth\crcr\unhbox0}}} \def\cprime{$'$}
  \def\cprime{$'$}
\providecommand{\bysame}{\leavevmode\hbox to3em{\hrulefill}\thinspace}
\providecommand{\MR}{\relax\ifhmode\unskip\space\fi MR }
\providecommand{\MRhref}[2]{%
  \href{http://www.ams.org/mathscinet-getitem?mr=#1}{#2}
}
\providecommand{\href}[2]{#2}

\vskip1in

\end{document}